\documentclass[12pt]{amsart}
\usepackage{amsfonts, amssymb, latexsym, hyperref, xcolor, tikz,tikz-cd,transparent}
\usepackage{ytableau, dynkin-diagrams}
\usepackage{empheq}

\newtheorem{theorem}{Theorem}[section]
\newtheorem{proposition}[theorem]{Proposition}
\newtheorem{lemma}[theorem]{Lemma}

\newtheorem{assumption}[theorem]{Assumption}

\newtheorem*{claim*}{Claim}
\newtheorem{corollary}[theorem]{Corollary}
\newtheorem{Main Conjecture}[theorem]{Main Conjecture}
\newtheorem{conjecture}[theorem]{Conjecture}

\theoremstyle{definition}
\newtheorem{definition}[theorem]{Definition}

\theoremstyle{remark}

\newtheorem{example}[theorem]{Example}

\newtheorem{remark}[theorem]{Remark}
\theoremstyle{plain}

\newcommand\complexes{{\mathbb C}}
\newcommand\integers{{\mathbb Z}}
\newcommand\naturals{{\mathbb N}}

\newcommand{\cellsize}{12}
\newlength{\cellsz} 
\newsavebox{\cell}
\sbox{\cell}{\begin{picture}(\cellsize,\cellsize)
\put(0,0){\line(1,0){\cellsize}}
\put(0,0){\line(0,1){\cellsize}}
\put(\cellsize,0){\line(0,1){\cellsize}}
\put(0,\cellsize){\line(1,0){\cellsize}}
\end{picture}}
\newcommand\cellify[1]{\def\thearg{#1}\def\nothing{}%
\ifx\thearg\nothing
\vrule width0pt height\cellsz depth0pt\else
\hbox to 0pt{\usebox{\cell} \hss}\fi%
\vbox to \cellsz{
\vss
\hbox to \cellsz{\hss$#1$\hss}
\vss}}
\newcommand\tableau[1]{\vtop{\let\\\cr
\baselineskip -16000pt \lineskiplimit 16000pt \lineskip 0pt
\ialign{&\cellify{##}\cr#1\crcr}}}

\newcommand{\gap}{\hspace{1in} \\ \vspace{-.2in}}

\hypersetup{colorlinks, linkcolor={red!80!black},
citecolor={blue!80!black}, urlcolor={blue!80!black}}

\newcommand{\excise}[1]{}

\newcommand{\grid}{%
\draw (1,1) -- (1,4);
\draw (2,1) -- (2,4);
\draw (0,2) -- (3.,2);
\draw (0,3) -- (3,3);
}
\tikzset{ball node/.style={draw,circle,inner sep=.1em,minimum size=2.5ex}}
\def\single(#1,#2)#3{\node[ball node] (A) at (#1.5,#2.5) {#3};}
\def\double(#1,#2)#3#4{\node[ball node](A) at (#1.25,#2.75) {#3};
                       \node[ball node] (B) at (#1.75,#2.25) {#4};}

\begin{document}
\pagestyle{plain}
\title{RSK as a linear operator}
\author{Ada Stelzer}
\author{Alexander Yong}
\address{Dept.~of Mathematics, U.~Illinois at Urbana-Champaign, Urbana, IL 61801, USA} 
\email{astelzer@illinois.edu, ayong@illinois.edu}
\date{December 11, 2024}

\maketitle

\vspace{-.3in}
\begin{abstract}
The Robinson--Schensted--Knuth correspondence (RSK) is a bijection between nonnegative integer matrices and
pairs of Young tableaux. 
We study it as a linear operator on the coordinate ring of matrices, proving results about its diagonalizability, eigenvalues, trace, and determinant. Our criterion for diagonalizability involves the $ADE$ classification of Dynkin diagrams, as well as the diagram for $E_9$. 
\end{abstract}

\section{Introduction}

\subsection{Background}
This paper is devoted to linear algebraic questions about the \emph{Robinson-Schensted-Knuth correspondence} (RSK), an important combinatorial algorithm. RSK can be interpreted as the transition operator between the ``representation-theoretic'' and ``obvious'' bases of the vector space of polynomial functions on matrices. Examples of transition matrices between such bases of vector spaces include:
\begin{itemize}
\item \emph{Kostka matrices} between the Schur and monomial bases of symmetric polynomials~\cite{ECII};
\item  \emph{Symmetric group character tables} between the irreducible character basis and the indicator function basis of class functions
\cite{Fulton, Serre}; and 
\item \emph{Kazhdan-Lusztig matrices} between the Kazhdan-Lusztig basis and the standard basis of a Hecke algebra \cite{KL}.
\end{itemize}
These matrices are of significant interest, and are all 
related to RSK.\footnote{See Stanley \cite[7.12]{ECII}, Ram \cite{Ram}, and Ariki \cite{Ariki} for instances of the respective connections.} 
Recognizing the centrality of RSK in combinatorial representation theory, we initiate a parallel study of the RSK transition matrix itself. 

A \emph{partition} $\lambda=(\lambda_1\geq \lambda_2\geq \ldots \geq \lambda_{\ell}\geq 0)$
is a weakly decreasing sequence of $\ell$ nonnegative integers. Identify $\lambda$ with its \emph{Young diagram}, a configuration of $\ell$ rows of left-justified boxes with $\lambda_i$ boxes in row $i$.
A \emph{semistandard Young tableau} is a filling of $\lambda$ with positive integers that weakly increase, left-to-right, along rows and strictly increase, top-to-bottom, along columns. If  $\lambda=(4,2,1)$, then $\tableau{1 &1 &2 & 2\\ 2 & 3 \\ 3}$ is one such tableau (drawn in English notation).  Let ${\sf SSYT}(\lambda,m)$ be the set of such tableaux taking values in $[m]:=\{1,2,\ldots,m\}$.

Fix $m,n\in {\mathbb N}:=\{0,1,2,\ldots\}$ and let ${\sf Mat}_{m,n}({\mathbb N})$ be the set of $m\times n$ matrices with entries from 
${\mathbb N}$. RSK is usually described as a bare set bijection
\[{\rm RSK}:{\sf Mat}_{m, n}(\naturals)\to \bigcup_{\lambda} {\sf SSYT}(\lambda,m)\times {\sf SSYT}(\lambda,n),\]
where the union is over all partitions $\lambda$ with at most $\min\{m, n\}$ rows. In Section~\ref{subsec:RSK} we recall one way to exhibit RSK via a combinatorial algorithm. 
The combinatorics of this bijection is well-studied, see e.g., the books \cite{Fulton, ECII} and references therein.

Our analysis of RSK is motivated by its equivalence to the first fundamental theorem of invariant theory for general linear groups (see \cite{Howe}).
Denote the coordinate ring of the space
${\sf Mat}_{m,n}({\mathbb C})$ of $m\times n$ complex matrices by
\[R_{m,n}:={\mathbb C}[z_{ij}]_{1\leq i\leq m, 1\leq j\leq n}.\]
As a $\complexes$-vector space, $R_{m,n}$ has two bases of interest. 
One is the ``obvious'' \emph{monomial basis},  
\[\left\{z^{\alpha}:=\prod_{i,j}z_{ij}^{\alpha_{i,j}}:[\alpha_{i,j}]\in {\sf Mat}_{m,n}({\mathbb N})\right\},\] 
where $\alpha=[\alpha_{i,j}]$ is an ``exponent matrix''. We will identify a monomial $z^{\alpha}$ 
with $\alpha$.
The second basis is the ``representation-theoretic'' \emph{bitableau basis} of Doubilet--Rota--Stein \cite{Rota}. It
was used, by \cite{Rota} and \cite{Procesi} respectively, to prove 
the first and second fundamental theorems of invariant theory for general linear groups over arbitrary commutative rings. 
Elements of the bitableaux basis are certain products of determinants $[P|Q]$ indexed by pairs $(P,Q)\in {\sf SSYT}(\lambda,m)\times {\sf SSYT}(\lambda,n)$; the definition is in Section~\ref{subsec:straight}.

Consequently, RSK may be interpreted as an operator
\[{\sf RSK}:R_{m,n}\to R_{m,n},\] 
by linearly extending the map 
\[\quad \quad z^{\alpha}\mapsto [P|Q] \quad (\text{where } (P, Q) := {\rm RSK}(\alpha)).\] 
Although ${\sf RSK}$ is an operator on an infinite-dimensional vector space, it decomposes as 
a direct sum of finite-dimensional operators. Let $R_{m, n, d}$ denote the vector space spanned by all degree-$d$ monomials in $R_{m, n}$. Then 
$R_{m, n} = \bigoplus_{d\geq 0}R_{m, n, d}$, 
and since ${\sf RSK}$ preserves degree (Lemma~\ref{lemma:degpreserve}) it splits as a direct sum of the restrictions ${\sf RSK}_{m, n, d}$ of ${\sf RSK}$ to $R_{m, n, d}$. 

We were led to investigate the linear operator {\sf RSK} by Bruns--Conca--Raicu--Varbaro's 
\cite[Question~4.2.8]{Bruns}, which asserts that little is known about it and asks, e.g., about its eigenvectors and eigenvalues. 
Our results concern the eigenvalues, diagonalizability, determinant, and trace of the matrices ${\sf RSK}_{m, n, d}$.

To state our diagonalizability result, define ${\mathcal G}_{m,n,d}$ to be the graph consisting of three paths of lengths $m$, $n$ and $d$ adjoined at one node in a ``$\bot$'' shape (so $|{\mathcal G}_{m,n,d}| = m+n+d-2$). Now, the following statement summarizes some of our major conclusions:

\begin{theorem}\label{thm:sampler}
	Let $m, n, d\in\naturals$. 
	\begin{itemize}
		\item[(I)](Theorem~\ref{thm:maindiag}) The matrix ${\sf RSK}_{m, n, d}$ is diagonalizable if and only if $d\leq 3$, or 
       		 $\mathcal{G}_{m, n, d}$ is a Dynkin diagram of ``ADE type'' $A_k$, $D_k$, $E_6$, $E_7$, $E_8$ (equivalently, $\frac{1}{m}+\frac{1}{n}+\frac{1}{d} > 1$), or $E_9$; see Figure~\ref{fig:dynkin}.
		\item[(II)](Theorem~\ref{thm:radicals}) The characteristic polynomial of ${\sf RSK}_{m, n, d}$ is not solvable by radicals whenever $m, n\geq 3$ and $d\geq 4$.
		\item[(III)](Theorem~\ref{thm:detperiod}) Fix $d$ and let $r$ be minimal such that $2^r>d$. The function $\det {\sf RSK}_{m, n, d}$ has period $2^r$ in both $m$ and $n$, i.e.,
			\[\det {\sf RSK}_{m, n, d} = \det {\sf RSK}_{m+2^r, n, d} = \det {\sf RSK}_{m, n+2^r, d}.\]
		\item[(IV)](Theorem~\ref{thm:tracepoly}) For fixed $d$, the trace of ${\sf RSK}_{m, n, d}$ is a polynomial in $O(m^dn^d)$.
	\end{itemize}
\end{theorem}
\begin{figure}
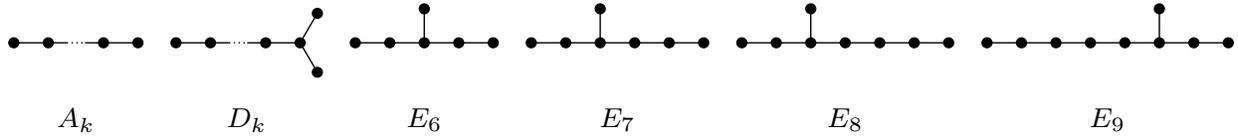

    \begin{center}
    \resizebox{6.5in}{0.3in}{$\underbracket[0pt]{\vphantom{\dynkin D{}}\dynkin A{}}_{A_k}$
    $\underbracket[0pt]{\dynkin D{}}_{D_k}$
    $\underbracket[0pt]{\vphantom{\dynkin D{}}\dynkin E6}_{E_6}$
    $\underbracket[0pt]{\vphantom{\dynkin D{}}\dynkin E7}_{E_7}$
    $\underbracket[0pt]{\vphantom{\dynkin D{}}\dynkin E8}_{E_8}$
    $\underbracket[0pt]{\vphantom{\dynkin D{}}\begin{dynkinDiagram}[ordering=Kac]E9
    \end{dynkinDiagram}}_{E_9}$}
    \end{center}
    \caption{The Dynkin diagrams corresponding to diagonalizable matrices ${\sf RSK}_{m, n, d}$.}
    \label{fig:dynkin}
\end{figure}
In Theorem~\ref{thm:sampler}(I), the inequality holds when $m=0,n=0$ or $d=0$ in the sense that $1/0=\infty$. The diagram $E_9$ indexes the special case where $(m, n, d) \in\{(2, 3, 6), (3, 2, 6)\}$. 
Our proofs use Theorem~\ref{thm:suffstab}, which concerns the further restriction of ${\sf RSK}$ to \emph{weight spaces} described below. The weight space arguments also yield formulas for 
the determinant and trace of ${\sf RSK}_{m, n, d}$ more efficient than the na\"ive algorithms.

\subsection{Weight spaces and RSK-commuting maps}
A pair of vectors $(\sigma, \pi)\in\naturals^m\times\naturals^n$ has \emph{degree} $d$ if 
$d=|\sigma| = |\pi|$, where $|\sigma| := \sum_{i=1}^m \sigma_i$. 
The \emph{weight space} 
$R_{m,n,\sigma,\pi}\subseteq R_{m,n, d}$ 
is the subspace spanned by degree-$d$ monomials $z^{\alpha}$ such that 
\begin{eqnarray}
\sum_j \alpha_{i,j}=\sigma_i, \ 1\leq i\leq m;\label{margin1} \\ 
\sum_i \alpha_{i,j}=\pi_j, \ 1\leq j\leq n \label{margin2}.
\end{eqnarray}
Equivalently, we say $\alpha$ is a \emph{contingency table} with row margins $\sigma$ and column margins $\pi$.\footnote{Thus a formula for $\dim_{\mathbb C} R_{m,n,\sigma,\pi}$ is unknown in general. Its computation is a $\#P$-complete problem \cite{Dyer} when one assumes the input data $\sigma,\pi$ are encoded in binary.} Now,
\begin{equation}
\label{eqn:thedirectsums}
R_{m,n, d}=\bigoplus_{\sigma,\pi:|\sigma|=|\pi|=d} R_{m,n, \sigma,\pi}.
\end{equation}
Although $R_{m, n}$ and $R_{m, n, d}$ are both $GL:=GL_m \times GL_n$ representations, the individual weight spaces are only representations of the maximal torus $T_m\times T_n\subseteq GL$. 
Our usage of the term ``weight space" is consistent with that in Lie theory. 

The \emph{content} of a Young tableau $T$ is the vector $(c_1,c_2,\ldots)$ such that $T$ contains $c_i$ $i$'s.  
Lemma~\ref{lemma:bibasis} states that the \emph{standard bitableaux} $[P|Q]$, where $P$ has content $\sigma$ and $Q$ has content $\pi$, form a linear basis of $R_{m,n,\sigma,\pi}$. 
Thus the restriction ${\sf RSK}_{m, n, \sigma, \pi}$ of ${\sf RSK}$ to $R_{m, n, \sigma, \pi}$ is well-defined. 
After reordering the basis, the matrix ${\sf RSK}_{m,n, d}$ is block diagonal with each block
a matrix ${\sf RSK}_{m,n,\sigma,\pi}$. Hence, it suffices to study ${\sf RSK}_{m,n,\sigma,\pi}$.

\begin{example}\label{exa:11-11}
    Let $m = n = 2$ and $\sigma = \pi = (1, 1)$. $R_{2, 2, \sigma, \pi}$ is two-dimensional, spanned by the monomials $\{z_{11}z_{22}, z_{12}z_{21}\}$.
The subset of the bitableaux basis spanning this space is
    \ytableausetup{aligntableaux=center, boxsize=1em}
    \[\left\{\left[ \ \begin{ytableau} 1 & 2 \end{ytableau}\ \big\vert\ \begin{ytableau} 1 & 2 \end{ytableau}\ \right],  \left[\ \begin{ytableau} 1\\ 2 \end{ytableau}\ \bigg\vert\  \begin{ytableau} 1 \\ 2 \end{ytableau}\ \right]\right\} = \left\{z_{11}z_{22}, \begin{vmatrix}z_{11} & z_{12}\\ z_{21} & z_{22}\end{vmatrix}\right\}.\]
    
    Now, ${\sf RSK}(z_{11}z_{22})=z_{11}z_{22}$ and 
    ${\sf RSK}(z_{12}z_{21})=\begin{vmatrix}z_{11} & z_{12}\\ z_{21} & z_{22}\end{vmatrix}=z_{11}z_{22}-z_{12}z_{21}$. Thus ${\sf RSK}_{2, 2, \sigma, \pi}$ is represented by 
 $\begin{bmatrix}1 & 1\\ 0 & -1\end{bmatrix}$.
A basis of eigenvectors is $\{z_{11}z_{22}, z_{11}z_{22}-2z_{12}z_{21}\}$, with eigenvalues $1$ and $-1$ respectively.
\end{example}

In order to uncover properties of ${\sf RSK}$, we aim to understand 
weight-pairs $(\sigma, \pi)$ and $(\widetilde\sigma, \widetilde\pi)$ such that the matrices ${\sf RSK}_{m, n, \sigma, \pi}$ and ${\sf RSK}_{\widetilde m, \widetilde n, \widetilde\sigma, \widetilde\pi}$ are similar. 
Define a linear map $\psi:R_{m,n,\sigma,\pi}\to R_{\widetilde m, \widetilde n,\widetilde \sigma, \widetilde \pi}$ to be \emph{RSK-commuting} if 
\[\psi \cdot {\sf RSK}_{m,n,\sigma,\pi} = {\sf RSK}_{\widetilde m, \widetilde n, \widetilde \sigma ,\widetilde \pi} \cdot \psi.\] 
An RSK-commuting isomorphism $\psi$ exists if and only if ${\sf RSK}_{m,n,\sigma,\pi}\sim{\sf RSK}_{\widetilde m, \widetilde n, \widetilde \sigma ,\widetilde \pi}$. In practice, we prove a given 
$\psi$ is RSK-commuting by analyzing its action on the monomial basis, checking its commutation with the \emph{combinatorial} algorithm ${\rm RSK}$.\footnote{For an example of a non-combinatorial RSK-commuting isomorphism, see Remark~\ref{remark:noncombo-ex}.}
For example, simple RSK-commuting isomorphisms, given in Lemma~\ref{lemma:easyRSKcommuting}, show that ${\sf RSK}_{m, n, \sigma, \pi}$ is determined by $(\sigma, \pi)$ alone, and we henceforth drop the $m$ and $n$ in the notation.
Let $\ell(\sigma)$ be the number of entries in $\sigma$.  Lemma~\ref{lemma:easyRSKcommuting} also allows us to make the following harmless assumption throughout this paper:

\begin{assumption}\label{ass:basic}
All weights $\sigma$ and $\pi$ have only nonzero entries, and $\ell(\sigma)\leq\ell(\pi)$. If $\ell(\sigma) = \ell(\pi)$ then $\sigma$ is ordered before 
$\pi$ lexicographically (i.e., $\sigma_i>\pi_i$ for the least $i$ such that $\sigma_i\neq\pi_i$).
\end{assumption}

Much of this paper is built around RSK-commuting isomorphisms that arise from variable multiplication. These isomorphisms are characterized by Theorem~\ref{thm:suffstab} below.

\begin{definition}\label{def:stabmap}
	Let $\sigma\in\naturals^m$ and $\pi\in\naturals^n$, and for indices $k\in[m]$ and $\ell\in[n]$ consider the standard basis vectors $\vec{e}_k\in\naturals^m$ and $\vec{e}_\ell\in\naturals^n$. 
	Then multiplication by $z_{k\ell}$ is a linear map
	\[\psi^{\sigma, \pi}_{k\ell}: R_{\sigma, \pi}\to R_{\sigma+\vec{e}_k, \pi+\vec{e}_\ell}.\]
\end{definition}

\begin{theorem}\label{thm:suffstab}
	Let $(\sigma, \pi)\in\naturals^m\times\naturals^n$ be a degree-$d$ weight pair and fix $(k, \ell)\in[m]\times[n]$.  
	The multiplication map $\psi^{\sigma, \pi}_{k\ell}$ is an RSK-commuting isomorphism if and only if
	\begin{equation}\label{eqn:stabineq}
		\sigma_k + \pi_\ell \geq d.
	\end{equation}
\end{theorem}

To operationalize Theorem~\ref{thm:suffstab}, we introduce the following poset:
\begin{definition}
	The \emph{variable-multiplication weight poset} $\mathcal{P}$ is a poset on weight pairs $(\sigma, \pi)$, graded by degree $d = |\sigma| = |\pi|$. It is defined by the covering relations
	\[(\sigma, \pi)\prec (\sigma+\vec{e}_k, \pi+\vec{e}_\ell)\quad \textrm{if}\ \  \sigma_k+\pi_\ell\geq d.\]
\end{definition} 

\begin{definition}
	A degree-$d$ weight pair $(\sigma, \pi)\in\naturals^m\times\naturals^n$ is \emph{reduced} if
	\[\max\{\sigma_i|1\leq i\leq m\}+\max\{\pi_j|1\leq j\leq n\}\leq d.\]
	In other words, $(\sigma, \pi)$ is reduced if it is a minimal element of $\mathcal{P}$.
\end{definition}

\begin{corollary}\label{cor:redunique}
	The connected component of any weight pair $(\sigma, \pi)$ in $\mathcal{P}$ contains a unique minimal element, the \emph{reduction} $(\sigma^{red}, \pi^{red})$ of $(\sigma, \pi)$, and
	${\sf RSK}_{\sigma^{red}, \pi^{red}} = {\sf RSK}_{\sigma,\pi}$.
\end{corollary}

Corollary~\ref{cor:redunique} allows us to restrict attention to ${\sf RSK}_{\sigma, \pi}$ where $(\sigma, \pi)$ is a reduced pair. In fact, Theorem~\ref{thm:blockbuild} presents a fully explicit decomposition 
of ${\sf RSK}_{m, n, d}$ as a direct sum of these ${\sf RSK}_{\sigma, \pi}$. This decomposition is central to the proofs of all parts of Theorem~\ref{thm:sampler}.

\subsection{Organization}
In Section~\ref{sec:prelim} we review ${\rm RSK}$ and bitableaux before
deducing basic consequences, such as the fact that 
$\det {\sf RSK}_{\sigma,\pi}\in \{\pm 1\}$ (Proposition~\ref{prop:det}).

Section~\ref{sec:main} discusses RSK-commuting isomorphisms, proving Theorem~\ref{thm:suffstab} and Corollary~\ref{cor:redunique}. These results are obtained by analysis of ${\rm RSK}$ and contingency tables.
We use these isomorphisms to explicitly decompose ${\sf RSK}_{m, n, d}$ as a direct sum of matrices ${\sf RSK}_{\sigma, \pi}$ where $(\sigma, \pi)$ is reduced (the ``Block decomposition theorem'', Theorem~\ref{thm:blockbuild}).

Section~\ref{sec:examples} presents three infinite families of reduced pairs $(\sigma, \pi)$ for which we can describe entries of ${\sf RSK}_{\sigma, \pi}$ more explicitly. We call these families 
\emph{permutation weights}, \emph{voting weights}, and \emph{triangular weights}, and use them in later sections to obtain further results. 

Section~\ref{sec:eigen} concerns eigenvalues of ${\sf RSK}$. We use voting weights with Corollary~\ref{cor:redunique} to prove Theorem~\ref{thm:rootsofunity}, which states that every root of unity appears an an eigenvalue of 
${\sf RSK}$. Proposition~\ref{prop:pm1} states the rational eigenvalues of ${\sf RSK}_{\sigma,\pi}$ can only be $\pm1$.  
Conjecture~\ref{conj:complex} states that if $\ell(\sigma), \ell(\pi)\geq 3$, ${\sf RSK}_{\sigma,\pi}$ has a non-real complex eigenvalue. 
For triangular weights all eigenvalues of ${\sf RSK}_{\sigma, \pi}$ are $\pm1$ (Proposition~\ref{thm:trieigen}), and Conjecture~\ref{conj:nontri} states that on all other reduced weight spaces ${\sf RSK}$ contains a non-rational eigenvalue.  Theorem~\ref{thm:radicals} shows that the characteristic polynomial of ${\sf RSK}_{m, n, d}$ is not generally solvable by radicals.

In Section~\ref{sec:diag} we characterize triples $(m, n, d)$ such that ${\sf RSK}_{m, n, d}$ is diagonalizable (Theorem~\ref{thm:maindiag}). 
The proof uses Theorem~\ref{thm:blockbuild} and the fact that ${\sf RSK}_{\sigma, \pi}$ is diagonalizable for triangular weights $(\sigma, \pi)$ (Proposition~\ref{prop:tridiag}).

In Sections~\ref{sec:det} and \ref{sec:trace} we consider the determinant and trace of ${\sf RSK}_{m, n, d}$. 
Theorem~\ref{thm:detperiod} and Theorem~\ref{thm:tracepoly} are analogous and use Theorem~\ref{thm:blockbuild} to describe, for fixed $d$, the periodic nature of $\det {\sf RSK}_{m, n, d}$ 
and polynomial growth of ${\rm Tr}\ {\sf RSK}_{m, n, d}$ respectively.

Finally, in Section~\ref{sec:tables} we compile some tables of data for reduced pairs $(\sigma,\pi)$.

\section{Preliminaries}\label{sec:prelim}
 
\subsection{RSK} \label{subsec:RSK}

We recall the RSK correspondence, following the standard treatment found in \cite[Section~7.11]{ECII} with one difference of convention. 
If our RSK algorithm associates a matrix to $(P, Q)$, then the RSK of \cite{Fulton} and \cite{ECII} associates it to $(Q, P)$. 

Given a semistandard tableau $P$ of shape $\lambda$, the row insertion of an integer $p\geq 1$, denoted $P\leftarrow p$, is defined as follows. 
Write $P = (P_1,\dots, P_{\ell(\lambda)})$, where $P_i$ is the $i$th row of~$P$. 
If $p$ is larger than all labels in $P_1$, then $P\leftarrow p$ is
the same as $P$ with $\begin{ytableau} p \end{ytableau}$ adjoined to the end of $P_1$. Otherwise, consider the smallest $p'>p$ appearing
in $P_1$. Let $P_1^*$ be $P_1$ with that $\begin{ytableau} p'\end{ytableau}$ replaced by $\begin{ytableau} p\end{ytableau}$ and define 
$P\leftarrow p$ 
to be $(P^*_1, \overline P\leftarrow p')$, where $\overline P = (P_2, P_3,\dots, P_{\ell(\lambda)})$.

Next, we define the insertion of a \emph{biletter} $(p|q)$ (an ordered pair of integers $p,q\geq 1$) into a pair of semistandard tableaux $(P,Q)$ of common
shape $\lambda$. We denote this operation by 
\[(P,Q)\leftarrow (p|q).\] 
First we compute $P\leftarrow p$, whose shape is the same as $P$ except with a new corner box added.
Then define $Q^\uparrow$ to be $Q$ with $\begin{ytableau} q\end{ytableau}$ placed in that same corner. Now $(P,Q)\leftarrow (p|q)$
is defined to be $(P\leftarrow p, Q^\uparrow)$.

Next, suppose $\alpha\in {\sf Mat}_{m,n}({\mathbb N})$. We record a sequence of biletters by reading the entries of $\alpha$
down the columns from left to right. We record each $\alpha_{i,j}$ as
$\left(\underbrace{ii\dots i}_{\alpha_{i,j}}\bigg |\underbrace{jj\dots j}_{\alpha_{i, j}}\right)$. The \emph{biword} of $\alpha$, denoted ${\sf biword}(\alpha)$, 
is the concatenation of all these biletters (written with extraneous brackets and commas removed). For example, if  
\[\alpha=\begin{bmatrix}
0 & 3 & 2 \\ 1 & 2 & 0 \\ 
2 & 0 & 2
\end{bmatrix}, {\sf biword}(\alpha)= (233111221133|111222223333).\]
Finally, we define ${\rm RSK}(\alpha)$ to be the result of inserting the biletters of 
\begin{equation}
\label{eqn:Oct7245}
{\sf biword}(\alpha)= (p_1p_2\dots p_d|q_1q_2\dots q_d)
\end{equation}
successively starting with $(\emptyset,\emptyset)$. That is, we compute
\[(P,Q)=\left(\cdots \left(\left((\emptyset,\emptyset)\leftarrow (p_1|q_1)\right)\leftarrow (p_2|q_2)\right)\leftarrow 
(p_3|q_3)\cdots\right).\]
The reader can check that in our running example,
\[\setlength{\delimitershortfall}{-5pt}
{\rm RSK}(\alpha)=\left( \ 
\begin{ytableau}
1 &1 & 1 & 1 &1 &3 &3 \\
2 &2 & 2\\
3 & 3\end{ytableau}\ ,
\ \begin{ytableau}
1 & 1 & 1 & 2 & 2 & 3 & 3 \\
2 & 2 & 2\\
3 &3\end{ytableau} \ \right).\]

We do not use the explicit inverse ${\rm RSK}$ map in this paper, but we give a brief description. Given 
\[(P,Q)\in {\sf SSYT}(\lambda,m)\times {\sf SSYT}(\lambda,n),\] 
search for the largest label of $Q$, which must appear in some corner 
${\sf c}$ of $\lambda$. If there are multiple instances of this largest label, pick the rightmost one. Remove this box, giving a tableau $Q^{\downarrow}$ and set $q$ to be the label in ${\sf c}$. 
Now, in $P$, we reverse insert the label $x$ in ${\sf c}$. That is, we first remove that box $\begin{ytableau} x\end{ytableau}$. If ${\sf c}$ appears in the first row of $\lambda$, then set $p=x$ and output the biletter $(p|q)$. 
Otherwise, in the previous row, find the rightmost label $x'$ such that $x'<x$. Replace that $x'$ with $x$. If
this $x'$ is in the first row, set $p=x'$ and output $(p|q)$. Otherwise, we continue by reverse inserting $x'$
in the previous row, eventually resulting in a tableau $P^{\downarrow}$. Now repeat the same process with 
$(P^{\downarrow},Q^\downarrow)$, continuing until the common shape is $\emptyset$. This gives a biword which
corresponds to a matrix, as above.

\subsection{Bitableaux and straightening}\label{subsec:straight}
We recall the bitableau basis of $R_{m, n}$ referenced in the introduction. Let $\Delta_1,\dots, \Delta_N$ be a sequence of minors of the generic $m\times n$ matrix
\[Z=[z_{ij}]_{1\leq i\leq m,1\leq j\leq n}.\]
We may assume that that the respective sizes of the minors are weakly decreasing. We encode the product
\[\Delta_1\dots\Delta_N\in R_{m, n}\]
as a pair of (not necessarily semistandard) Young tableaux $(P, Q)$, where the $c$-th columns (from the left) of $P$ and $Q$ are filled by the row and column indices of $\Delta_c$ respectively.\footnote{In \cite{Bruns} the pairs $(P, Q)$ are displayed differently, so the indices of $\Delta_c$ come from rows of $P$ and $Q$ instead.} 
When $P$ and $Q$ are both semistandard, we call the corresponding product of minors a \emph{standard bitableau} and denote it $[P|Q]$.  

\begin{example}
The following product of minors is a standard bitableau in $R_{4, 4}$:
\[\setlength{\delimitershortfall}{-5pt}
\begin{vmatrix}
z_{11} & z_{12 } & z_{14 }\\
z_{21} & z_{22} & z_{24 }\\
z_{41} & z_{42 } & z_{44 }\\
\end{vmatrix}
\begin{vmatrix}
z_{12} & z_{13 }\\
z_{32} & z_{33}
\end{vmatrix} z_{22} z_{23}
= \left[\ \begin{ytableau}1 & 1 & 2 & 2\\ 2 & 3 \\ 4\end{ytableau} \ \bigg\vert\ \begin{ytableau}1 & 2 & 2 & 3\\ 2 & 3 \\ 4\end{ytableau}\ \right].\]
The simplest product of minors that is not a standard bitableau is
\[\setlength{\delimitershortfall}{-5pt}
z_{21}z_{12} \leftrightarrow\left(\ \begin{ytableau} 2 & 1\end{ytableau}, \begin{ytableau}1 & 2\end{ytableau} \ \right).\]
\end{example}

The straightening law of \cite{Rota} allows one to write any product of minors $\Delta_1 \cdots\Delta_N$ in terms of standard bitableaux. 
The main results of the theory are summarized as follows:

\begin{theorem}[{\cite[Theorem 3.2.1]{Bruns}}]\gap \label{thm:mainstraight}
\begin{itemize}
\item[(I)] The standard bitableaux $[P|Q]$ form a ${\mathbb C}$-linear basis of $R_{m,n}$.
\item[(II)] If a product of minors $\Delta\Delta'$ is not a standard bitableau then
\begin{equation}
\label{eqn:straight}
\Delta\Delta'=\sum_{i} d_i \Theta_i\Theta_i',  \ d_i\in {\mathbb Z}-\{0\}
\end{equation}
where each $\Theta_i\Theta_i'$ is a standard bitableau.
\item[(III)] Every product of minors can be expressed as a $\integers$-linear combination of standard bitableaux by
successive application of (II).
\end{itemize}
\end{theorem}

We do not inspect the precise rule for (\ref{eqn:straight}) except to say that it can be derived from the Pl\"ucker relations
for the coordinate ring of the Grassmannian. 
The main form of straightening we use is recorded below as Proposition~\ref{prop:straighteasy}.

\begin{proposition}\label{prop:straighteasy}
Every monomial $z^{\alpha}\in R_{m,n}$ is an $\integers$-linear combination of the standard bitableaux.
\end{proposition}
\noindent\emph{Proof of Proposition~\ref{prop:straighteasy}:} View each variable $z_{ij}$ as a $1\times 1$ minor.
Hence $z^{\alpha}$ is a product of minors, corresponding to a pair of possibly non-semistandard $1$-row tableaux. Now apply Theorem~\ref{thm:mainstraight}(III).\qed

\subsection{Basic consequences}\label{sec:basic}
With ${\rm RSK}$ and Theorem~\ref{thm:mainstraight} stated, our restriction of ${\sf RSK}$ to weight spaces is now easily justified.

\begin{lemma}\label{lemma:degpreserve}
	${\sf RSK}$ is a degree-preserving map, i.e., it restricts to a linear operator on the finite-dimensional
	vector space $R_{m,n,d}$ spanned by monomials $z^{\alpha}$ of degree $d$.
\end{lemma}
\begin{proof}
	By definition, if $z^\alpha$ has degree $d$ then $\alpha$ (viewed as a biword) contains $d$ biletters, so ${\rm RSK}(\alpha) = (P,Q)$ for some tableaux $P, Q$ of size $d$. 
	Then $[P|Q]$ is a degree-$d$ homogeneous polynomial, so ${\sf RSK}$ preserves degree as claimed.
\end{proof}

\begin{lemma}\label{lemma:bibasis}
	The standard bitableaux $[P|Q]$ with content $(\sigma, \pi)$ form a linear basis of $R_{\sigma,\pi}$.
\end{lemma}
\begin{proof}
	By Theorem~\ref{thm:mainstraight}(I), it suffices to show that if $[P|Q]$ has content $(\sigma, \pi)$, then $[P|Q]$ lies in $R_{m, n, \sigma, \pi}$. 
	Every monomial appearing in a minor of $Z$ uses the same row and column indices.  
	The row and column indices used in the $c$-th minor of $[P|Q]$ are determined by the $c$-th columns of $P$ and $Q$, respectively. 
	If $[P|Q]$ has content $(\sigma, \pi)$, it follows that every monomial $z^\alpha$ in the expansion of $[P|Q]$ contains 
	$\sigma_i$ variables from row $i$ and $\pi_j$ variables from column $j$ of $Z$ (counted with multiplicity).
	Thus $[P|Q]$ lies in $R_{m, n, \sigma, \pi}$ as claimed. 
\end{proof}

\begin{proposition}
	The columns of ${\sf RSK}_{\sigma, \pi}$ have sum $0$ or $1$. The unique column with sum $1$ corresponds to the $\alpha\in{\sf Cont}_{\sigma, \pi}$ such that ${\rm RSK}(\alpha)$ has only one row.
\end{proposition}
\begin{proof}
Order the monomial basis $\{z^{\alpha}\}$ of $R_{m,n,\sigma, \pi}$. The columns of ${\sf RSK}_{\sigma, \pi}$ record the 
coefficients of 
${\sf RSK}(z^{\alpha})=[P|Q]$
expanded back into monomials. The sum of these coefficients is 
obtained by setting $z_{ij}=1$. If the common shape of $P,Q$ has more than one row, then one of the minors
in $[P|Q]$ vanishes under this substitution. Otherwise the monomial evaluates to  $1$. 
\end{proof}

\begin{proposition}\label{prop:det}
$\det({\sf RSK}_{m,n,d}), \det({\sf RSK}_{\sigma,\pi})\in \{\pm 1\}$.
\end{proposition}
\begin{proof}
Clearly $M={\sf RSK}_{\sigma,\pi}\in{\sf Mat}_{m, n}(\integers)$, so $\det(M)\in\integers$. Since $\det(M)\det(M^{-1})=1$, the claim for $M$ follows by showing that 
${\sf RSK}^{-1}_{\sigma,\pi}$ is also an integer matrix. Indeed, ${\sf RSK}^{-1}$ is computed by first taking $z^{\alpha}$
and expressing it as an $\integers$-linear combination of standard bitableaux via Proposition~\ref{prop:straighteasy}. The bitableaux $[P|Q]$ appearing in this linear combination have content $(\sigma,\pi)$ by Lemma~\ref{lemma:bibasis},
and the coefficients then form a column of the matrix ${\sf RSK}^{-1}_{\sigma, \pi}$. 
With this, the conclusion $\det({\sf RSK}_{m,n,d})\in \{\pm 1\}$ follows from (\ref{eqn:thedirectsums}).
\end{proof}

\subsection{Notational conventions and an example}
Before continuing, we establish some conventions for weight vectors that will be used throughout the rest of the paper. 
Recall that a \emph{weight pair} of \emph{degree} $d$ is a tuple $(\sigma, \pi)\in\naturals^m\times\naturals^n$ such that $d = |\sigma| = |\pi|$. 
We often write $(\sigma, \pi)$ in the abbreviated form $(\sigma_1\sigma_2\dots\sigma_m, \pi_1\pi_2\dots\pi_n)$. For example, we write $(21, 111)$ as shorthand for $((2, 1), (1, 1, 1))$. 
The \emph{length} $\ell(\sigma)$ of $\sigma$ is the number of entries it contains. 
Lowercase Greek letters generally denote nonnegative integer tuples: $\sigma$, $\pi$, $\tau$, and $\rho$ are weight vectors; $\lambda$ is a partition; 
$\alpha$ and $\beta$ are exponent matrices. Two exceptions are the minimal polynomial $\mu_M(t)$ of a matrix $M$ and the Kronecker delta function $\delta_{i, j}$. 

Now let $\alpha\in{\sf Mat}_{m, n}$ be a contingency table and write ${\rm RSK}(\alpha) = (P, Q)$. The \emph{shape} of $\alpha$ (or the corresponding monomial $z^\alpha$) is the common shape $\lambda$ of $P$ and $Q$. 
Since we index monomials $z^\alpha$ in $R_{\sigma, \pi}$ by their exponent matrices $\alpha$, we also use contingency tables to index the rows and columns of ${\sf RSK}_{\sigma, \pi}$. 
To be fully explicit, the entry ${\sf RSK}_{\sigma, \pi}(\beta, \alpha)$ is defined to be $[z^\beta]{\sf RSK}(z^\alpha)$, the coefficient of $z^\beta$ in the bitableau associated to $z^\alpha$ by ${\sf RSK}$. 
We order the exponent matrices of monomials in $R_{\sigma, \pi}$ lexicographically, as follows: 

\begin{definition}\label{def:contorder}
	Let ${\sf Cont}_{\sigma, \pi}$ denote the set of all contingency tables with row margins $\sigma$ and column margins $\pi$; see \eqref{margin1} and \eqref{margin2}.
	We order ${\sf Cont}_{\sigma, \pi}$ by placing $\alpha$ before $\alpha'$ if $z^\alpha > z^{\alpha'}$ in 
	the lexicographic ordering where $z_{11} > z_{21} > \dots > z_{m1} > z_{12} > \dots > z_{mn}$.
\end{definition}

\begin{remark}\label{rem:order}
	Our lexicographic ordering is chosen to agree with the order in which biletters are inserted during ${\rm RSK}$. 
	More precisely, $z_{ij} > z_{k\ell}$ if and only if $(i|j)$ is listed before $(k|\ell)$ in the ordering of biletters used in (\ref{eqn:Oct7245}).
\end{remark}

\begin{remark}[Algorithms]
We used two algorithms to compute ${\sf RSK}_{\sigma,\pi}$.  Both begin by determining the monomial basis of $R_{\sigma,\pi}$, which amounts to generating ${\sf Cont}_{\sigma, \pi}$. 
One method is to find all pairs of semistandard Young tableaux $(P,Q)$ of content $(\sigma, \pi)$.

\noindent
\emph{Algorithm A:} Take each $z^{\alpha}$ in the (ordered) monomial basis of $R_{\sigma, \pi}$, compute ${\sf RSK}(z^{\alpha})=[P|Q]$, and expand the bitableau. Then ${\sf RSK}_{\sigma, \pi}(\beta, \alpha) = [z^\beta][P|Q]$.

\noindent
\emph{Algorithm B:} Express $z^{\alpha}$ as a $\integers$-linear combination of standard bitableaux by straightening: 
\[z_{\alpha}=\sum_{P,Q} c_{P,Q}[P|Q].\] 
Then we have ${\sf RSK}^{-1}_{\sigma, \pi}(\beta, \alpha) = c_{P, Q}$, where $\beta = {\rm RSK}^{-1}(P, Q)$. Compute the matrix inverse.

Algorithm A requires numerous ($\dim R_{\sigma, \pi}$-many) expansions of bitableaux into a large number of monomials. Algorithm B avoids these expansions, but depends on a straightening algorithm.
\end{remark}

\begin{example}\label{exa:111-111}
	Let $(\sigma, \pi)=(111, 111)$.
   	The vector space $R_{\sigma,\pi}$ is six-dimensional, with ordered monomial basis 
   	\[\{z_{11}z_{22}z_{33}, z_{11}z_{32}z_{23}, z_{21}z_{12}z_{33}, z_{21}z_{32}z_{13}, z_{31}z_{12}z_{23}, z_{31}z_{22}z_{13}\}.\]
   	The bitableau basis of $R_{\sigma,\pi}$ is
   	\ytableausetup{aligntableaux=center, boxsize=0.9em}
   	\begin{align*}
		\ & \left\{\left(\begin{ytableau} 1 & 2 & 3 \end{ytableau}, \begin{ytableau} 1 & 2 & 3 \end{ytableau}\right), \left(\begin{ytableau} 1 & 2 \\ 3 \end{ytableau}, \begin{ytableau} 1 & 2 \\ 3 \end{ytableau}\right), 
		\left(\begin{ytableau} 1 & 3 \\ 2 \end{ytableau}, \begin{ytableau} 1 & 3 \\ 2 \end{ytableau}\right), \left(\begin{ytableau} 1 & 3 \\ 2 \end{ytableau}, \begin{ytableau} 1 & 2 \\ 3 \end{ytableau}\right), 
		\left(\begin{ytableau} 1 & 2 \\ 3 \end{ytableau}, \begin{ytableau} 1 & 3 \\ 2 \end{ytableau}\right), \left(\begin{ytableau} 1\\ 2 \\ 3 \end{ytableau}, \begin{ytableau} 1 \\ 2 \\ 3 \end{ytableau}\right)\right\}\\
		&= \left\{z_{11}z_{22}z_{33}, \begin{vmatrix} z_{11} & z_{13}\\ z_{31} & z_{33}\end{vmatrix}z_{22}, \begin{vmatrix} z_{11} & z_{12}\\ z_{21} & z_{22}\end{vmatrix}z_{33}, 
		\begin{vmatrix} z_{11} & z_{13}\\ z_{21} & z_{23}\end{vmatrix}z_{32}, \begin{vmatrix} z_{11} & z_{12}\\ z_{31} & z_{32}\end{vmatrix}z_{23}, 
		\begin{vmatrix}z_{11} & z_{12} & z_{13}\\ z_{21} & z_{22} & z_{23}\\ z_{31} & z_{32} & z_{33}\end{vmatrix}\right\}.
	\end{align*}

	The basis sets above are ordered such that ${\sf RSK}$ preserves the ordering; for example, 
	\[\setlength{\delimitershortfall}{-5pt}
	{\sf RSK}(z_{31}z_{12}z_{23})=\left(\begin{ytableau} 1 & 2 \\ 3 \end{ytableau}, \begin{ytableau} 1 & 3 \\ 2 \end{ytableau}\right).\] 
	The reader can use Algorithm A to check that, with respect to the ordered bases,
	\[{\sf RSK}_{111,111}=\begin{bmatrix}
	1 & 1 & 1 & 0 & 0 & 1\\
	0 & 0 & 0 & 1 & 1 & -1\\
	0 & 0 & -1 & 0 & 0 & -1\\
	0 & 0 & 0 & -1 & 0 & 1\\
	0 & 0 & 0 & 0 & -1 & 1\\
	0 & -1 & 0 & 0 & 0 & -1\\
    	\end{bmatrix}.\]
	The characteristic and minimal polynomials are, respectively,
	\[p_{{\sf RSK}_{111, 111}}(t)=(t-1)(t+1)^2(t^3+2t^2+1), \ 
	\mu_{{\sf RSK}_{111, 111}}(t)=(t-1)(t+1)(t^3+2t^2+1).\]
	The non-integer eigenvalues are roots of a nontrivial cubic.\footnote{Those roots being 
	$-\frac{\theta}{6}-\frac{8}{3\theta}-\frac{2}{3}$, $\frac{\theta}{12}+\frac{4}{3\theta}-\frac{2}{3}\pm \frac{\sqrt{3}i}{2}(-\frac{\theta}{6}+\frac{8}{3\theta})$ where $\theta=\sqrt[3]{172+12\sqrt{177}}$.}
	The integer eigenvectors are
    	\[(z_{11}z_{22}z_{33}, z_{11}z_{22}z_{33}-2z_{12}z_{21}z_{33}, z_{12}z_{23}z_{31}-z_{13}z_{21}z_{32}),\]
	with eigenvalues $(1, -1, -1)$ respectively, but the other three eigenvectors have unpleasant coordinates. 
	This basis of eigenvectors shows that ${\sf RSK}_{111, 111}$ is diagonalizable.
\end{example}

\section{RSK-commuting maps and proof of Theorem~\ref{thm:suffstab} and Corollary~\ref{cor:redunique}; the Block decomposition theorem}\label{sec:main}

The goal of this section is to provide proofs for our main results on RSK-commuting isomorphisms. This leads to our main consequence, the Block decomposition theorem for ${\sf RSK}_{m,n,d}$ (Theorem~\ref{thm:blockbuild}).
 
Our arguments follow a shared method.
Let $\psi:R_{m, n, \sigma, \pi}\to R_{\widetilde m, \widetilde n, \widetilde\sigma, \widetilde\pi}$ be a linear map induced by a set map 
$\psi:{\sf Cont}_{\sigma, \pi}\to{\sf Cont}_{\widetilde\sigma, \widetilde\pi}$. Let $\widetilde\alpha := \psi(\alpha)$, $(P, Q) := {\rm RSK}(\alpha)$, and $(\widetilde P, \widetilde Q) := {\rm RSK}(\widetilde\alpha)$. 
We check that $\psi$ is RSK-commuting by showing that for each $\beta\in{\sf Cont}_{\sigma, \pi}$ we have $[z^\beta][P|Q] = [z^{\widetilde\beta}][\widetilde P|\widetilde Q]$. 
We begin by justifying our Assumption~\ref{ass:basic}.

\begin{lemma}\label{lemma:easyRSKcommuting} 
Let $\sigma\in {\mathbb N}^m, \pi\in {\mathbb N}^n$.
\begin{itemize}
	\item[(I)] If $\sigma^+=(\sigma,0)\in {\mathbb N}^{m+1}$ then
	\[{\sf RSK}_{m,n,\sigma,\pi}={\sf RSK}_{m+1,n,\sigma^+,\pi}.\]
	\item[(II)] Let $\sigma^+=(\sigma_1,\ldots,\sigma_k,0,\sigma_{k+1},\ldots,\sigma_m)\in {\mathbb N}^{m+1}$ . Then
	\[{\sf RSK}_{m,n,\sigma,\pi}={\sf RSK}_{m+1,n,\sigma^+,\pi}.\]
	\item[(III)] ${\sf RSK}_{m,n,\sigma,\pi}\sim{\sf RSK}_{n,m,\pi,\sigma}$.
\end{itemize}
\end{lemma}
\begin{proof}[Proof of Lemma~\ref{lemma:easyRSKcommuting}]
(I): This follows by the RSK-commuting isomorphism that sends 
\[z^{\alpha}\in R_{m,n,\sigma,\pi}\mapsto
z^{\alpha}\in R_{m+1,n,\sigma^+,\pi}.\]
In this case $[P|Q] = [\widetilde P|\widetilde Q]$, so it is clear that the same monomials appear in each expansion.

(II):  Let $\alpha\in {\sf Cont}_{\sigma,\pi}$. Define $\alpha^+\in{\sf Cont}_{\sigma^+, \pi}$ to be $\alpha$ with a row of $0$'s inserted after row $k$. The map $\psi:\alpha\mapsto\alpha^+$
is a bijection between ${\sf Cont}_{\sigma, \pi}$ and ${\sf Cont}_{\sigma^+, \pi}$. If ${\rm RSK}(\alpha) = (P, Q)$, then 
\[{\rm RSK}(\alpha^+)=(P^+,Q),\] where $P^+$ is
$P$ with each label $p$ shifted to $p+1$ for all $p>k$. Thus $[z^\beta][P|Q] = [z^{\beta^+}][P^+|Q]$, so $\psi$ is RSK-commuting.

(III): Define $\psi:{\sf Cont}_{\sigma, \pi}\to {\sf Cont}_{\pi, \sigma}$ by sending $\alpha$ to its transpose matrix $\alpha^t$.
Then $\psi$ is a bijection, and one of the symmetry properties of RSK (see e.g. \cite[pg. 40]{Fulton}) is that 
\[{\rm RSK}(\alpha)=(P, Q) \iff {\rm RSK}(\alpha^t)=(Q,P).\] 
Since $[z^\beta][P|Q] = [z^{\beta^t}][Q|P]$, we conclude that $\psi$ is RSK-commuting.
\end{proof}

Lemma~\ref{lemma:easyRSKcommuting} justifies Assumption~\ref{ass:basic} about our weights $(\sigma, \pi)$. 
Parts (I) and (III) show that ${\sf RSK}_{m, n, \sigma, \pi}$ is determined by $(\sigma, \pi)$ alone, parts (II) and (III) allow us to freely assume that $\sigma$ and $\pi$ have no nonzero entries, and 
part (III) justifies the assumptions that $\ell(\sigma)\leq\ell(\pi)$ and that if $\ell(\sigma) = \ell(\pi)$ then $\sigma$ is ordered before $\pi$ lexicographically. 

\begin{example}\label{exa:111/111} 
Natural linear isomorphisms $R_{\sigma, \pi}\to R_{\widetilde\sigma, \widetilde\pi}$ may not be RSK-commuting.
	Let 
	\[(\sigma,\pi)=(21,111)\text{ and } (\widetilde \sigma,\widetilde\pi)=(12, 111).\]
	Then
	\[{\sf RSK}_{\sigma,\pi}=
	\left[\begin{matrix}
	1 & 1 & 0 \\
	0 & 0 & 1 \\
	0 & -1 & -1 
	\end{matrix}\right] \text{ and }\ 
	{\sf RSK}_{\widetilde \sigma,\widetilde \pi}=
	\left[\begin{matrix}
	1 & 1 & 1 \\
	0 & -1 & 0 \\
	0 &0 & -1 
	\end{matrix}\right].\]
	Although swapping rows $1$ and $2$ of the contingency tables induces a linear isomorphism $\psi:R_{\sigma, \pi}\to R_{\widetilde\sigma, \widetilde\pi}$, this map is not RSK-commuting. 
	Indeed, the matrices above are not similar. 
	They have different eigenvalues, $(1,\frac{-1\pm i\sqrt{3}}{2})$ for the former matrix and $(1,-1,-1)$ for the latter. 
	Thus there is no RSK-commuting isomorphism between $R_{\sigma,\pi}$ and $R_{\widetilde\sigma,\widetilde\pi}$. 
	For cases where swapping parts of contingency tables \emph{is} RSK-commuting, see Corollary~\ref{cor:permbij}.
\end{example}

Next, we work towards a proof of Theorem~\ref{thm:suffstab}.
The proof uses several technical lemmas, the first two of which (Lemmas~\ref{lemma:bumpineq} and \ref{lemma:divineq}) concern the combinatorics of ${\rm RSK}$. 

\begin{definition}
	Let $\alpha\in {\sf Mat}_{m, n}$ be a contingency table of shape $\lambda$ (viewed as a biword), and let ${\rm RSK}(\alpha) = (P, Q)$. 
	Then the $c$-th \emph{bump chain} of $\alpha$ is the sequence of biletters
	\[{\sf chain}_c(\alpha) := ((p_1|q_1),\dots,(p_s|q_s))\]
	that are inserted into the $c$-th box (counted from the left) in the first row of $\lambda$ when computing $(P,Q)$ from $\alpha$ via the insertion algorithm presented in Section~\ref{subsec:RSK}.
\end{definition}

\begin{example}\label{exa:chains}
	Let $\alpha = \begin{bmatrix}0 & 1 & 2\\ 1 & 1 & 0\\ 2 & 1 & 0\end{bmatrix}$, so $\setlength{\delimitershortfall}{-5pt}
	{\rm RSK}(\alpha) = \left(\ \begin{ytableau}1 & 1 & 1 & 3\\ 2 & 2 & 3\\ 3\end{ytableau},\  \begin{ytableau}1 & 1 & 1 & 2\\ 2 & 2 & 3\\ 3\end{ytableau} \ \right).$ 
	The four bump chains for $\alpha$ are as follows:
	\begin{align*}
		{\sf chain}_1(\alpha) = ((2|1), (1|2)), &\quad {\sf chain}_2(\alpha) = ((3|1), (2|2), (1|3)),\\
		{\sf chain}_3(\alpha) = ((3|1), (1|3)), &\quad {\sf chain}_4(\alpha) = ((3|2)).
	\end{align*}
\end{example}

\begin{remark}\label{rem:matrixball}
	Bump chains have a graphical interpretation via Fulton's \emph{matrix-ball} realization of ${\rm RSK}$ \cite[Section 4.2]{Fulton}. 
	Indeed, ${\sf chain}_c(\alpha)$ is the set of positions in $\alpha$ containing a ball labelled ``$c$" in Fulton's construction. 
	We will not review the matrix-ball construction in detail, as the usual insertion algorithm suffices for our arguments, but those already familiar may find it helpful for visualization. 
	The matrix-ball diagram illustrating Example~\ref{exa:chains} is displayed below:
	
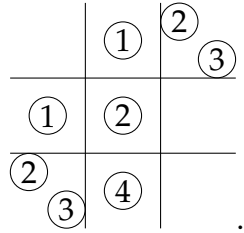
\begin{center}
	\begin{tikzpicture}
		\grid
		\double(2,3){2}{3}
		\double(0,1){2}{3}
		\single(0,2){1}
		\single(1, 1){4}
		\single(1, 2){2}
		\single(1, 3){1}
	\end{tikzpicture}.
	\end{center}
\end{remark}

Define ${\sf value}_c(\alpha)$ to be the biletter $(p|q)$, where $p$ and $q$ are
respectively the labels of box $c$ in the first row of the $P$ and $Q$ tableaux of ${\rm RSK}(\alpha)$. The \emph{length} of a bump chain $C$ is $|C|$. 
Mildly abusing notation, we will identify a biletter $(p|q)$ with the exponent matrix of the variable $z_{pq}$, 
and a bump chain $C = ((p_1|q_1),\dots, (p_s|q_s))$ with the exponent matrix of the monomial  $z^C := z_{p_1q_1}z_{p_2q_2}\dots z_{p_sq_s}$. 
The lexicographic order on monomials in Definition~\ref{def:contorder} then defines lexicographic total orders on biletters and bump chains.
\begin{align*}
	(p|q) > (p'|q')&\iff z_{pq} > z_{p'q'},\\
	C > C' &\iff z^C > z^{C'}.
\end{align*}

\begin{remark}
	Our lexicographic order is not graded, so $z_{11} > z_{12}z_{21}$. 
	By convention, any $z_{pq}$ is larger (i.e. earlier) than $1$, which forces $C > C'$ whenever $C\supset C'$.
\end{remark}

\begin{proposition}\label{prop:basicbump}
    Let $\alpha\in{\sf Mat}_{m, n}$ be a contingency table of shape $\lambda$. Then
    \begin{enumerate}
	\item[(I)] The set $\{{\sf chain}_c(\alpha)\}_{c=1}^{\lambda_1}$ partitions the biletters of $\alpha$.
	\item[(II)] If 
	\[{\sf chain}_c(\alpha) = ((p_1|q_1),\dots, (p_s|q_s)),\] 
	then $p_i > p_j$ and $q_i < q_j$ whenever $i < j$ (i.e., the elements of bump chains form strict antidiagonal sequences in the matrix $\alpha$). 
	In particular, $(p_i|q_i)>(p_j|q_j)$.
	\item[(III)] If 
	\[{\sf chain}_c(\alpha) = ((p_1|q_1),\dots, (p_s|q_s)),\] 
	then ${\sf value}_c(\alpha) = (p_s|q_1)$.
	\item[(IV)] If $c\leq c'$, then ${\sf chain}_c(\alpha)\geq {\sf chain}_{c'}(\alpha)$ and ${\sf value}_c(\alpha)\geq{\sf value}_{c'}(\alpha)$. 
    \end{enumerate}
\end{proposition}
\begin{proof}
	All four statements are immediate from the definitions of ${\rm RSK}$ and ${\sf chain}_c(\alpha)$. 
\end{proof}

\begin{lemma}\label{lemma:bumpineq}
    Fix $k\in[m]$ and $\ell\in[n]$ and let $\alpha\in{\sf Mat}_{m, n}(\naturals)$ satisfy
    \begin{equation}\label{eqn:bumpineq}
	\alpha_{k, \ell} > \sum_{i>k, j<\ell}\alpha_{i, j} + \sum_{i'<k, j'>\ell}\alpha_{i', j'}.
    \end{equation}
    Let $c$ be maximal such that $(k|\ell)\in{\sf chain}_c(\alpha)$. Then ${\sf chain}_c(\alpha) = ((k|\ell))$.
\end{lemma}
\begin{proof}
	Take any $b$ such that $(k|\ell)\in{\sf chain}_b(\alpha)$. 
	By Proposition~\ref{prop:basicbump}(II), if ${\sf chain}_b(\alpha)$ contains another biletter $(i|j)$ then either $i>k$ and $j<\ell$, or $i<k$ and $j>\ell$. 
	When (\ref{eqn:bumpineq}) is satisfied, it follows that ${\sf chain}_b(\alpha) = ((k|\ell))$ for some $b$. 
	Then $b\leq c$ by the definition of $c$, so $((k|\ell))\geq {\sf chain}_c(\alpha)$ by Proposition~\ref{prop:basicbump}(IV). But the fact that $(k|\ell)\in{\sf chain}_c(\alpha)$ 
	implies ${\sf chain}_c(\alpha) \geq ((k|\ell))$. We conclude that ${\sf chain}_c(\alpha) = ((k|\ell))$ as desired.
\end{proof}

\begin{lemma}\label{lemma:divineq}
	Fix $k\in[m]$ and $\ell\in[n]$, let $\alpha\in{\sf Mat}_{m, n}(\naturals)$ of shape $\lambda$ satisfy
	\begin{equation}\label{eqn:divineq}
		\alpha_{k, \ell} > \sum_{i\neq k, j\neq \ell}\alpha_{i, j},
	\end{equation}
	and let $c$ be maximal such that ${\sf value}_c(\alpha) = (k|\ell)$ ($c$ exists by Lemma~\ref{lemma:bumpineq}). Then $c > \lambda_2$.
\end{lemma}
\begin{proof}
	By Proposition~\ref{prop:basicbump}(IV), $c$ counts $b\in[\lambda_1]$ such that ${\sf value}_b(\alpha)\geq (k|\ell)$. 
	Note that  $\lambda_2$ is bounded above by the number of bump chains in $\alpha$ of length at least $2$, since 
	at most one label $p$ from a bump chain can appear in any given row of the $P$-tableau, and no length-$1$ chain contributes to the second row of $P$.  
	It therefore suffices to demonstrate the following inequality:
	\begin{equation}\label{eqn:bumpchainineq}
		|\{b\in[\lambda_1]: {\sf value}_{b}(\alpha) \geq (k|\ell)\}| > |\{b\in[\lambda_1]: |{\sf chain}_{b}(\alpha)|\geq 2\}|.
	\end{equation}
	Let $L$ and $R$ denote the sets appearing on the left and right sides of (\ref{eqn:bumpchainineq}) respectively. For convenience of notation, let 
	\[C_b := {\sf chain}_{b}(\alpha) = ((p^b_1|q^b_1),\dots, (p^b_{s_b}|q^b_{s_b})).\]
	 We define
	\[X := \{b\in[\lambda_1]:|C_b|\geq 2\text{ and }(k|\ell)\in\{(p^b_1|q^b_{s_b}),(p^b_{s_b}|q^b_1)\}\}.\]
	We claim $X\subseteq L\cap R$. By definition $X\subseteq R$, so we show $X\subseteq L$. If $b\in X$, then $(k|\ell) = (p^b_1|q^b_{s_b})$ or $(p^b_{s_b}|q^b_1)$. 
	In the first case, by Proposition~\ref{prop:basicbump}(II) we must have $q^b_1 < \ell = q^b_{s_b}$ and $p^b_{s_b} < k = p^b_1$. 
	Proposition~\ref{prop:basicbump}(III) then shows that ${\sf value}_b(\alpha) = (p^b_{s_b}|q^b_1) \geq (k|\ell)$, so $b\in L$. 
	In the second case, by Proposition~\ref{prop:basicbump}(III) we see ${\sf value}_b(\alpha) = (k|\ell)$, so again $b\in L$ and the claim is proved. 
	The two cases are illustrated in Figure~\ref{fig:divineq}.
	\begin{figure}
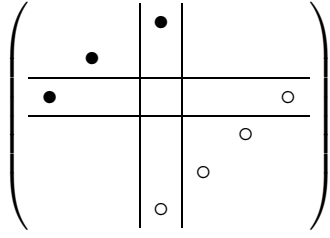

		\[\left(\begin{array}{ccc|c|ccc}
		&  &  & \bullet & & &\\
		& \bullet & & & &  &\\
		\hline
		\bullet & & & & & & \circ\\
		\hline
		& & & & & \circ &\\
		&  &  &  & \circ & &\\
	 	&  & & \circ & &  &
		\end{array}\right)\]
    	\caption{\label{fig:divineq} Two bump chains $C_b$, $C_{b'}$ with $b, b'\in X$. The chain of $\bullet$'s has $(p^b_1|q^b_{s_b}) = (k|\ell)$, 
	while the chain of $\circ$'s has $(p^{b'}_{s_{b'}}|q^{b'}_1) = (k|\ell)$.}
  	\end{figure}

	We complete the proof of the lemma by showing that 
	\[|L\setminus X| \geq \alpha_{k, \ell} > \sum_{i\neq k, j\neq \ell}\alpha_{i, j}\geq |R\setminus X|,\]
	since then $|L|-|X| = |L\setminus X| > |R\setminus X| = |R| - |X|$. 
	To show that $|L\setminus X|\geq \alpha_{k, \ell}$, note that by Proposition~\ref{prop:basicbump}(II) there are $\alpha_{k,\ell}$ distinct elements $b\in[\lambda_1]$ such that $(k|\ell)\in{\sf chain}_b(\alpha)$. 
	For any such $b$ we know that ${\sf chain}_b(\alpha)\geq ((k|\ell))\geq {\sf chain}_c(\alpha)$, so ${\sf value}_b(\alpha)\geq {\sf value}_c(\alpha) = (k|\ell)$ by Proposition~\ref{prop:basicbump}(IV). 
	Thus $b\in L$. If ${\sf chain}_{b}(\alpha) = ((k|\ell))$ then $b\notin X$ by the definition of $X$. Otherwise, Proposition~\ref{prop:basicbump}(II) implies that either $p_1 > k$ and $q_1 < \ell$, or $p_s < k$ and $q_s > \ell$,  
	so we again find that $b\notin X$. Thus $|L\setminus X| \geq \alpha_{k, \ell}$ as claimed.

	To show that $\sum_{i\neq k, j\neq \ell}\alpha_{i, j}\geq |R\setminus X|$, let $b\in R\setminus X$. Then at least one endpoint of ${\sf chain}_b(\alpha)$ is a biletter $(i|j)$ lying outside of row $k$ and column $\ell$, 
	which comes from some entry $\alpha_{i,j}$ of $\alpha$ with $i\neq k$, $j\neq\ell$. This completes the proof of (\ref{eqn:bumpchainineq}) and thus the proof of the lemma.
\end{proof}

The other two lemmas used in the proof of Theorem~\ref{thm:suffstab}, Lemmas~\ref{lemma:stabcondequiv} and ~\ref{lemma:stabbij}, concern the combinatorics of contingency tables.

\begin{definition}
	Fix $r,r'\in[m]$ and $c,c'\in [n]$. The \emph{swap matrix} 
	\[S_{(r, r'|c, c')}\in{\sf Mat}_{m, n}(\integers)\]
 	is the matrix with entries
	\[S_{(r, r'|c, c')}(i, j) = 
	\begin{cases}
		1 \quad \textrm{if }(i, j) \in\{(r, c), (r', c')\},\\
		-1 \quad \textrm{if }(i, j)\in\{(r, c'), (r', c)\},\\
		0 \quad \textrm{else}.
	\end{cases}\]
	Visually, swap matrices have the following form:
	\[\begin{bmatrix}
	 &  &  &  &  \\
	 & 1 & \dots & -1 &  \\
	 & \vdots & \ddots & \vdots  &  \\
	 & -1 & \dots & 1 & \\
	 &  &  &  &\end{bmatrix}\text{ or } 
	\begin{bmatrix}
	 &  &  &  &  \cr
	 & -1 & \dots & 1 &  \cr
	 & \vdots & \ddots & \vdots  &  \cr
	 & 1 & \dots & -1 & \cr
	 &  &  &  &\end{bmatrix}.\]
	A \emph{swap move} refers to adding a copy of $S_{(r, r'|c, c')}$ to $\alpha\in{\sf Cont}_{\sigma, \pi}$, where $r, r', c, c'$ are such that 
	the result has nonnegative entries, i.e., $\alpha+ S_{(r,r'|c,c')}\in {\sf Cont}_{\sigma, \pi}$.
\end{definition}

\begin{definition}\label{def:cantable}
	The \emph{canonical contingency table} $\alpha^0\in{\sf Cont}_{\sigma, \pi}$ is the unique element of ${\sf Cont}_{\sigma, \pi}$ such that ${\rm RSK}(\alpha^0)$ is a $1$-row bitableau.
\end{definition}

The following proposition motivates the use of swap moves and was originally proved by Diaconis--Gangiolli in \cite{Diaconis}. 
Although it is not needed in this paper, we give an argument using ${\rm RSK}$ which differs somewhat from their proof and may be of independent interest. 

\begin{proposition}[{\cite[Theorem 10.5]{Diaconis}}]\label{prop:swapconnect}
    The set ${\sf Cont}_{\sigma, \pi}$ is connected by swap moves. Also, $\alpha^0$ is the largest element of ${\sf Cont}_{\sigma, \pi}$ under the lexicographic ordering of Definition~\ref{def:contorder}.
\end{proposition}
\begin{proof}  
	We claim that any $\beta\in {\sf Cont}_{\sigma, \pi}$ is connected to the canonical contingency table $\alpha^0$ via a sequence of swap moves. 
	Indeed, if $\beta\neq\alpha^0$ then ${\rm RSK}(\beta)$ has more than one row. 
	Thus for some $c$ we have $|{\sf chain}_c(\beta)|\geq 2$. Let $(p_1|q_1)$ and $(p_2|q_2)$ be the first two elements of ${\sf chain}_c(\beta)$. 
	Then $\beta+S_{(p_1, p_2|q_1, q_2)} > \beta$ in the lexicographic total ordering on ${\sf Cont}_{\sigma, \pi}$. 
	But ${\sf Cont}_{\sigma, \pi}$ must have a unique maximal element. 
	Thus we have shown that $\alpha^0$ is this maximal element, and that every $\beta\in{\sf Cont}_{\sigma, \pi}$ is connected to $\alpha^0$ by swap moves.
\end{proof}

\begin{lemma}\label{lemma:stabcondequiv}
	Let $(\sigma, \pi)$ be a degree-$d$ weight pair and let $\alpha\in{\sf Cont}_{\sigma, \pi}$. Then
	\[\alpha_{k, \ell} - \sum_{i\neq k, j\neq \ell}\alpha_{i, j} = \sigma_k+\pi_\ell - d.\]
	In particular, $\alpha_{k, \ell} > \sum_{i\neq k, j\neq \ell}\alpha_{i, j}$ for all $\alpha\in{\sf Cont}_{\sigma, \pi}$ if and only if $\sigma_k + \pi_\ell > d$.
\end{lemma}
\begin{proof}
	This follows from inclusion-exclusion: simply rearrange the terms in the expression
	\[d = \sigma_k + \pi_\ell - \alpha_{k, \ell} + \sum_{i\neq k, j\neq\ell} \alpha_{i,j}.\qedhere\]
\end{proof}

\begin{lemma}\label{lemma:stabbij}
    Let $(\sigma, \pi)$ be a degree-$d$ weight pair. Every monomial in $R_{\sigma, \pi}$ is divisible by $z_{k\ell}$ if and only if
    \[\sigma_k + \pi_\ell > d.\]
\end{lemma}
\begin{proof}
	If $\sigma_k+\pi_\ell > d$, then $\alpha_{k, \ell} > 0$ for all $\alpha\in{\sf Cont}_{\sigma, \pi}$ by Proposition~\ref{lemma:stabcondequiv} and thus every monomial $z^\alpha\in R_{\sigma, \pi}$ is divisible by $z_{k\ell}$.
	Conversely, suppose that $\sigma_k + \pi_\ell \leq d$. Then by Proposition~\ref{lemma:stabcondequiv} there exists $\alpha\in{\sf Cont}_{\sigma, \pi}$ such that
	\begin{equation}\label{eqn:fail}
		\alpha_{k,\ell}\leq \sum_{i\neq k, j\neq\ell}\alpha_{i, j}. 
	\end{equation}
	We argue by induction on $\alpha_{k, \ell}$. If $\alpha_{k, \ell} = 0$ then we are done. Otherwise, (\ref{eqn:fail}) implies that for some $i\neq k$ and $j\neq\ell$ the matrix $\beta:=\alpha + S_{(k, i|j, \ell)}$ 
	lies in ${\sf Cont}_{\sigma, \pi}$. Then $\beta_{k, \ell} = \alpha_{k, \ell}-1$ and the entries of $\beta$ still satisfy (\ref{eqn:fail}). This completes the proof.
\end{proof}

\begin{proof}[Proof of Theorem~\ref{thm:suffstab}]
	The variable-multiplication map $\psi^{\sigma, \pi}_{k\ell}$ is clearly well-defined and injective. 
	Moreover, $\psi^{\sigma, \pi}_{k\ell}$ is surjective if and only if every monomial $z^\alpha\in R_{\sigma+\vec{e}_k, \pi+\vec{e}_\ell}$ is divisible by $z_{k\ell}$. 
	This occurs if and only if (\ref{eqn:stabineq}) holds by Lemma~\ref{lemma:stabbij}.

	It remains to show that (\ref{eqn:stabineq}) is sufficient for $\psi^{\sigma, \pi}_{k\ell}$ to commute with RSK. Let $\alpha\in{\sf Cont}_{\sigma, \pi}$ 
	and define 
	\[\tilde\alpha = \alpha + \vec{e}_k\otimes\vec{e}_\ell \text{\ (so $\psi^{\sigma, \pi}_{k\ell}(z^\alpha) = z^{\tilde\alpha}$).}\] 
	Then condition (\ref{eqn:stabineq}) and Proposition~\ref{lemma:stabcondequiv} imply that
	\[\widetilde\alpha_{k, \ell} > \sum_{i\neq k, j\neq\ell}\widetilde\alpha_{i, j},\]
	so $\widetilde\alpha$ satisfies the hypotheses of Lemmas~\ref{lemma:bumpineq} and ~\ref{lemma:divineq}.
	Let $\lambda$ and $\widetilde\lambda$ denote the shapes of $\alpha$ and $\widetilde\alpha$ respectively, and let 
	\[{\rm RSK}(\alpha)=(P,Q) \text{\ and  \ }{\rm RSK}(\tilde \alpha)
	=(\widetilde P, \widetilde Q).\]

	We need to show that 
	\[{\sf RSK}(\psi^{\sigma, \pi}_{k\ell}(z^\alpha)) = \psi^{\sigma, \pi}_{k\ell}({\sf RSK}(z^\alpha)),\] i.e., that $[\widetilde P|\widetilde Q] = z_{k\ell}[P|Q]$. 
	We claim first that $(P, Q)$ and $(\widetilde P, \widetilde Q)$ differ only in the first row. 
	Let $c$ be maximal such that $(k|\ell)\in {\sf chain}_{c}(\widetilde\alpha)$. Then Lemma~\ref{lemma:bumpineq} shows that ${\sf chain}_{c}(\widetilde\alpha) = ((k|\ell))$. 
	It is then straightforward from the algorithmic definition of ${\rm RSK}$ that 
	\[{\sf chain}_i(\widetilde\alpha) = \begin{cases}
		{\sf chain}_i(\alpha) & \text{if } 1\leq i < c,\\
		((k|\ell)) & \text{if } i = c,\\
		{\sf chain}_{i-1}(\alpha) & \text{if } c < i\leq \widetilde\lambda_1.
	\end{cases}\]
	This proves the claim.  
	Next, let $c'$ be maximal such that ${\sf value}_{c'}(\widetilde\alpha) = (k|\ell)$. By Lemma~\ref{lemma:divineq} we know that $c' > \widetilde\lambda_2$, 
	and since ${\sf value}_i(\widetilde\alpha) = (k|\ell)$ whenever $c \leq i\leq c'$ by Proposition~\ref{prop:basicbump}(IV) it follows that
	\[{\sf value}_i(\widetilde\alpha) = \begin{cases}
		{\sf value}_i(\alpha) & \text{if } 1\leq i< c',\\
		(k|\ell) & \text{if } i = c',\\
		{\sf value}_{i-1}(\alpha) & \text{if } c' < i \leq \widetilde\lambda_1.
	\end{cases}\]
	Now write $[P|Q]=\prod_{i=1}^{\lambda_1} \Delta_i$ and $[\widetilde P|\widetilde Q]=\prod_{i=1}^{\widetilde\lambda_1}\widetilde\Delta_i$ as products of minors.  
	Taken together, our computations of ${\sf chain}_i(\widetilde\alpha)$ and ${\sf value}_i(\widetilde\alpha)$ show that
	\[\widetilde\Delta_i = \begin{cases}
		\Delta_i & \text{if } 1\leq i <c',\\
		z_{k\ell} & \text{if } i = c',\\
		\Delta_{i-1} & \text{if } c' < i \leq \widetilde\lambda_1,
	\end{cases}\]
	from which we immediately see that $[\widetilde P|\widetilde Q] = z_{k\ell}[P|Q]$. This completes the proof.
\end{proof}

\begin{remark}
	If we consider RSK-commuting \emph{injections} rather than isomorphisms, an analogue of Theorem~\ref{thm:suffstab} with weaker hypotheses follows 
	by improving the bounds in Lemmas~\ref{lemma:bumpineq} and \ref{lemma:divineq}. For Lemma~\ref{lemma:bumpineq}, the right hand side of (\ref{eqn:bumpineq}) can be decreased to
	\[\max\left\{\sum_{i>k, j<\ell}\alpha_{i, j},\sum_{i'<k, j'>\ell}\alpha_{i', j'}\right\}\]
	by observing that the $(k|\ell)$ biletters that bump earlier biletters are the first to be themselves bumped by later biletters. 
	For Lemma~\ref{lemma:divineq}, one can improve (\ref{eqn:divineq}) by enlarging $X$ to include more of $L\cap R$. 
	As these strengthenings are unnecessary for our applications, we focus on RSK-commuting isomorphisms only.
\end{remark}

Our next goal is to prove Corollary~\ref{cor:redunique}. We introduce some notation to track the effects of applying multiple variable-multiplication isomorphisms.

\begin{definition}
	For a degree-$d$ weight pair $(\sigma, \pi)$, the \emph{growth potential matrix} $\mathbf{g}_{\sigma, \pi}\in{\sf Mat}_{m, n}(\integers)$ is the integer matrix with entries
	\[\mathbf{g}_{\sigma, \pi}(i, j) := \sigma_i + \pi_j-d.\]
\end{definition}

\begin{proposition}\label{prop:basicgrowth}
	Let $(\sigma, \pi)$ be a degree-$d$ weight pair.
	\begin{itemize}
		\item[(I)] Division by $z_{k\ell}$ is an RSK-commuting isomorphism $R_{\sigma, \pi}\to R_{\sigma-\vec{e}_k, \pi-\vec{e_\ell}}$ if and only if $\mathbf{g}_{\sigma, \pi}(k, \ell)$ is positive. 
		Multiplication by $z_{k\ell}$ is an RSK-commuting isomorphism $R_{\sigma, \pi}\to R_{\sigma+\vec{e}_k, \pi+\vec{e_\ell}}$ if and only if $\mathbf{g}_{\sigma, \pi}(k, \ell)$ is nonnegative.
		\item[(II)] The growth potential matrix of $(\sigma\pm\vec{e}_k, \pi\pm\vec{e}_\ell)$ is
		\[\mathbf{g}_{\sigma\pm\vec{e}_k, \pi\pm\vec{e}_\ell}(i, j) = \begin{cases}
			\mathbf{g}_{\sigma, \pi}(i, j)\pm 1 & \text{if } (i, j) = (k, \ell),\\
			\mathbf{g}_{\sigma, \pi}(i, j) & \text{if } i = k, j\neq\ell\text{ or } i\neq k, j=\ell,\\
			\mathbf{g}_{\sigma, \pi}(i, j)\mp 1 & \text{if } i\neq k, j\neq\ell.
		\end{cases}\]
		\item[(III)] If $\mathbf{g}_{\sigma, \pi}(k, \ell)$ and $\mathbf{g}_{\sigma, \pi}(k', \ell')$ are both positive, then $k = k'$ or $\ell = \ell'$. 
		The same conclusion holds if $\mathbf{g}_{\sigma, \pi}(k, \ell)$ and $\mathbf{g}_{\sigma, \pi}(k', \ell')$ are both nonnegative and $\ell(\sigma)$ or $\ell(\pi)$ is at least $3$.
	\end{itemize}
\end{proposition}
\begin{proof}
	Part (I) is a restatement of Theorem~\ref{thm:suffstab}, while part (II) is immediate from the definition of $\mathbf{g}_{\sigma, \pi}$. 
	For part (III), suppose $\mathbf{g}_{\sigma, \pi}(k, \ell), \mathbf{g}_{\sigma, \pi}(k', \ell') > 0$. Then by definition 
	\[(\sigma_k+\pi_\ell) + (\sigma_{k'}+\pi_{\ell'}) > d+d = |\sigma|+|\pi|,\]
	which is impossible unless $k=k'$ or $\ell = \ell'$. Similarly, if $\mathbf{g}_{\sigma, \pi}(k, \ell), \mathbf{g}_{\sigma, \pi}(k', \ell')\geq 0$, then  
	\[(\sigma_k+\pi_\ell) + (\sigma_{k'}+\pi_{\ell'}) \geq d+d = |\sigma|+|\pi|,\]
	which can only occur if $k = k'$, $\ell = \ell'$, or $\ell(\sigma) = \ell(\pi) = 2$.
\end{proof}

\begin{proof}[Proof of Corollary~\ref{cor:redunique}]
	We claim that the unique reduction $(\sigma^{red}, \pi^{red})$ of an arbitrary weight pair $(\sigma, \pi)$ is defined by
	\[\sigma^{red}_i = \sigma_i - \sum_{j=1}^n \max\{\mathbf{g}_{\sigma, \pi}(i, j), 0\},\quad \pi^{red}_j = \pi_j-\sum_{i=1}^m \max\{\mathbf{g}_{\sigma, \pi}(i, j), 0\}.\]
	If $(\sigma, \pi)$ is not reduced, then $\mathbf{g}_{\sigma, \pi}$ contains a positive entry by Proposition~\ref{prop:basicgrowth}(I). 
	By Proposition~\ref{prop:basicgrowth}(III), the positive entries of $\mathbf{g}_{\sigma, \pi}$ all lie in a single row or column. 
	It follows by Proposition~\ref{prop:basicgrowth}(II) that if $\mathbf{g}_{\sigma, \pi}(k, \ell) > 0$, then $\mathbf{g}_{\sigma-\vec{e}_k, \pi-\vec{e}_\ell}(k, \ell) = \mathbf{g}_{\sigma, \pi}(k, \ell)-1$ 
	and all other positive entries of $\mathbf{g}_{\sigma, \pi}$ and $\mathbf{g}_{\sigma-\vec{e}_k, \pi-\vec{e}_\ell}$ agree. 
	Thus division by all the variables in  
	\[z^\alpha := \prod_{i, j}z_{ij}^{\max\{\mathbf{g}_{\sigma, \pi}(i, j), 0\}}\]
	in any order defines an RSK-commuting isomorphism $R_{\sigma, \pi}\cong R_{\sigma^{red}, \pi^{red}}$. 
	Proposition~\ref{prop:basicgrowth}(I) guarantees that no other variable-division isomorphisms are possible in any step of this process, completing the proof. 
\end{proof}

\begin{example}\label{exa:reduction}
	Let 
	\[(\sigma, \pi) = (61, 232).\] 
	We compute the reduction 
	\[(\sigma^{red}, \pi^{red}) = (21, 111)\] 
	via Corollary~\ref{cor:redunique}. The growth potential matrix is
	\[\mathbf{g}_{61, 232} = \begin{bmatrix} 1 & 2 & 1\\ -4 & -3 & -4\end{bmatrix}.\]
	Corollary~\ref{cor:redunique} asserts that division by $z^{\max\{\mathbf{g}_{61, 232}, 0\}} = z_{11}z_{12}^2z_{13}$ is an RSK-commuting isomorphism $R_{61, 232}\cong R_{21, 111}$. 
	Figure~\ref{fig:reduction} shows the subset of the variable-multiplication poset $\mathcal{P}$ lying beneath $(61, 232)$. Corollary~\ref{cor:redunique} identifies
	$(21, 111)$ as its unique minimal element.
	
	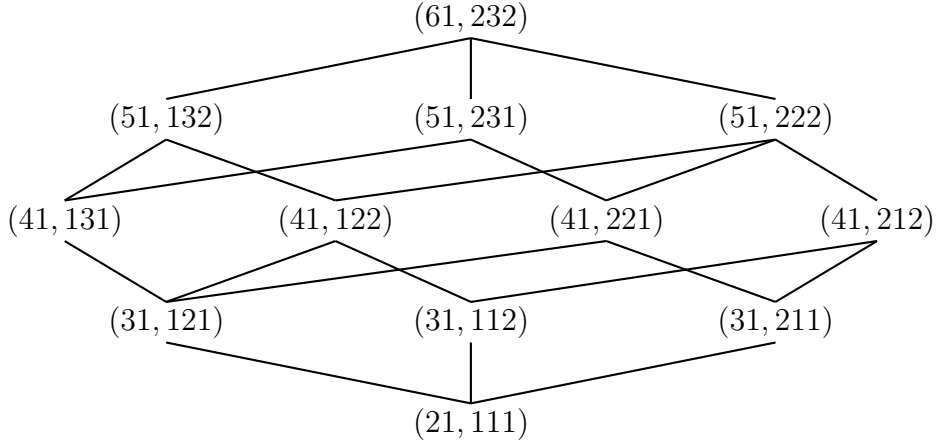
\begin{figure}
   	\begin{center}
    	\begin{tikzpicture}[scale = 0.45]
	        \draw[black, thick][-] (12, 11.4) -- (3, 9.6);
	        \draw[black, thick][-] (12, 11.4) -- (12, 9.6);
	        \draw[black, thick][-] (12, 11.4) -- (21, 9.6);

	       \draw[black, thick][-] (3, 8.4) -- (0, 6.6);
	       \draw[black, thick][-] (3, 8.4) -- (8, 6.6);
	       \draw[black, thick][-] (12, 8.4) -- (0, 6.6);
	       \draw[black, thick][-] (12, 8.4) -- (16, 6.6);
	       \draw[black, thick][-] (21, 8.4) -- (8, 6.6);
	       \draw[black, thick][-] (21, 8.4) -- (16, 6.6);
	       \draw[black, thick][-] (21, 8.4) -- (24, 6.6);

	       \draw[black, thick][-] (0, 5.4) -- (3, 3.6);
	       \draw[black, thick][-] (8, 5.4) -- (3, 3.6);
	       \draw[black, thick][-] (8, 5.4) -- (12, 3.6);
	       \draw[black, thick][-] (16, 5.4) -- (3, 3.6);
	       \draw[black, thick][-] (16, 5.4) -- (21, 3.6);
	       \draw[black, thick][-] (24, 5.4) -- (12, 3.6);
	       \draw[black, thick][-] (24, 5.4) -- (21, 3.6);

	       \draw[black, thick][-] (3, 2.4) -- (12, 0.6);
	       \draw[black, thick][-] (12, 2.4) -- (12, 0.6);
	       \draw[black, thick][-] (21, 2.4) -- (12, 0.6);

	        \draw (12, 12) node{$(61, 232)$};

	        \draw (3, 9) node{$(51, 132)$};
	        \draw (12, 9) node{$(51, 231)$};
	        \draw (21, 9) node{$(51, 222)$};

	        \draw (0, 6) node{$(41, 131)$};
                  \draw (8, 6) node{$(41, 122)$};
	        \draw (16, 6) node{$(41, 221)$};
	        \draw (24, 6) node{$(41, 212)$};

	        \draw (3, 3) node{$(31, 121)$};
	        \draw (12, 3) node{$(31, 112)$};
	        \draw (21, 3) node{$(31, 211)$};

	        \draw (12, 0) node{$(21, 111)$};
    	\end{tikzpicture}
    	\end{center}
    	\caption{\label{fig:reduction} The poset of weights $(\widetilde\sigma, \widetilde\pi)$ reachable from $(61, 232)$ by RSK-commuting variable-division isomorphisms.}
  	\end{figure}
\end{example}

\begin{corollary}\label{cor:redweuse}
	Let $(\sigma, \pi)$ be a degree-$d$ weight pair.
	\begin{itemize}
		\item[(I)] If $\min\{\ell(\sigma), \ell(\pi)\} \leq 1$, then \[\sigma^{red}=\vec{0}, \pi^{red} = \vec{0}.\]
		\item[(II)] If $\ell(\sigma) = \ell(\pi) = 2$, then for $a := \min\{\sigma_1, \sigma_2, \pi_1, \pi_2\}$ we have 
		\[\sigma^{red} = (a, a) = \pi^{red}.\]
		\item[(III)] If $\ell(\sigma) = 2$ and $\sigma_2 = 1$, then 
		\[\sigma^{red} = (d-1, 1) \text{\ and $\pi^{red} = (1,1,\dots,1)$.}\]
	\end{itemize}
\end{corollary}
\begin{proof}
	These all follow from Corollary~\ref{cor:redunique} after writing out each $\mathbf{g}_{\sigma, \pi}$ explicitly.
\end{proof}

With the theory just developed, we can now decompose ${\sf RSK}_{m,n,d}$ into a direct sum of
blocks ${\sf RSK}_{\sigma, \pi}$ indexed by reduced weight pairs $(\sigma, \pi)$. This is the Block decomposition
theorem below (Theorem~\ref{thm:blockbuild}). 
The decomposition requires enumeration of degree-$d$ weight pairs with a given reduction, which we do using Theorem~\ref{thm:suffstab}.

\begin{definition}
	The \emph{growth potential} of a reduced pair $(\sigma, \pi)$ is 
	\[g_{\sigma, \pi} := |\{(k, \ell):\mathbf{g}_{\sigma, \pi}(k, \ell) = 0\}|.\]
\end{definition}

\begin{corollary}\label{cor:redcount}
	Fix $d$ and let $(\sigma, \pi)$ be a nonzero reduced pair of degree $d'\leq d$.  Let $A_{\sigma, \pi}(d)$ denote the number of degree-$d$ weight pairs $(\widetilde\sigma, \widetilde\pi)$
	(satisfying Assumption~\ref{ass:basic}) that reduce to $(\sigma, \pi)$.
	\begin{itemize}
		\item[(I)] If $\ell(\pi)\geq 3$ then 
		\[A_{\sigma, \pi}(d) = 
		\begin{cases}
			{(d-d')+(g_{\sigma, \pi}-1)\choose g_{\sigma, \pi}-1} & \text{if } g_{\sigma, \pi}\geq1,\\
			\delta_{d,d'} & \text{if } g_{\sigma, \pi} = 0.
		\end{cases}\]
		\item[(II)] If $\ell(\sigma) = \ell(\pi) = 2$ then 
		\[A_{\sigma, \pi}(d) = 4(d-d')+\delta_{d,d'}.\]
	\end{itemize} 
\end{corollary}
\begin{proof}
	First suppose that $\ell(\pi) \geq 3$. Then by Proposition~\ref{prop:basicgrowth}(III), all nonnegative entries of $\mathbf{g}_{\sigma, \pi}$ lie in the same row or column. 
	By Proposition~\ref{prop:basicgrowth}(II), if $\mathbf{g}_{\sigma, \pi}(k, \ell)$ is nonnegative, 
	then the sets $\{(i, j):\mathbf{g}_{\sigma, \pi}(i, j)\geq 0\}$ and $\{(i, j):\mathbf{g}_{\sigma+\vec{e}_k, \pi+\vec{e}_\ell}(i, j)\geq 0\}$ are equal. 
	It follows from these facts and Proposition~\ref{prop:basicgrowth}(I) that the degree-$d$ weight pairs $(\widetilde\sigma, \widetilde\pi)$ reducing to $(\sigma, \pi)$ 
	are in bijection with degree-$(d-d')$ monomials in the $g_{\sigma, \pi}$ variables $\{z_{ij} : \mathbf{g}_{\sigma, \pi}(i, j) = 0\}$. 
	The formula in part (I) of the corollary statement is a textbook count of these monomials.

	If $\ell(\sigma) = \ell(\pi) = 2$, then by Corollary~\ref{cor:redweuse}(II) we have $(\sigma, \pi) = (aa, aa)$ where $a := d'/2$. 
	In this case Proposition~\ref{prop:basicgrowth}(III) does not apply (indeed, $\mathbf{g}_{aa, aa}$ is the zero matrix), so we instead compute $A_{aa, aa}(d)$ directly. 
	By Proposition~\ref{prop:basicgrowth}(I), if $(\widetilde\sigma, \widetilde\pi)$ reduces to $(aa, aa)$ then $\ell(\widetilde\sigma) = \ell(\widetilde\pi) = 2$. 
	Thus $\min\{\widetilde\sigma_1, \widetilde\sigma_2, \widetilde\pi_1, \widetilde\pi_2\} = a$ by Corollary~\ref{cor:redweuse}(II). 
	We enumerate these pairs by inclusion-exclusion on the sets $X_{\tau_k} = \{(\widetilde\sigma, \widetilde\pi) : \min\{\widetilde\sigma_1, \widetilde\sigma_2, \widetilde\pi_1, \widetilde\pi_2\} = \tau_k = a\}$, 
	where $\tau_k\in\{\widetilde\sigma_1, \widetilde\sigma_2, \widetilde\pi_1, \widetilde\pi_2\}$. By symmetry it suffices to enumerate 
	\[X_{\widetilde\sigma_1},\ X_{\widetilde\sigma_1}\cap X_{\widetilde\pi_1},\  X_{\widetilde\sigma_1}\cap X_{\widetilde\pi_2},\  X_{\widetilde\sigma_1}\cap X_{\widetilde\sigma_2}\cap X_{\widetilde\pi_1},\  
	\text{and}\ X_{\widetilde\sigma_1}\cap X_{\widetilde\sigma_2}\cap X_{\widetilde\pi_1}\cap X_{\widetilde\pi_2}.\] 

	We begin with $X_{\widetilde\sigma_1}$. The pairs $(\widetilde\sigma, \widetilde\pi)\in X_{\widetilde\sigma_1}$ are identified uniquely by the value $b = \widetilde\pi_1$, which must satisfy $b\geq a$ and $d-b\geq a$. 
	There are exactly $d-2a+1 = d-d'+1$ such integers, so $|X_{\widetilde\sigma_1}| = d-d'+1$. Of the pairs $(\widetilde\sigma, \widetilde\pi)\in X_{\widetilde\sigma_1}$, only the one with $b = a$ also lies in 
	$X_{\widetilde\pi_2}$. All other intersections listed above are empty unless $d = d'$, in which case they all contain the single pair $(aa, aa)$. 
	Inclusion-exclusion then yields the desired value $4(d-d')+\delta_{d, d'}$ for $A_{\sigma, \pi}(d)$. 
\end{proof}

\begin{theorem}[Block Decomposition Theorem]\label{thm:blockbuild}
	Fix $m, n, d\in\naturals$. Then we have the block matrix decomposition
	\[{\sf RSK}_{m, n, d} = \left({\rm Id}_1^{\oplus N_0(m, n, d)}\right)\oplus\left(\bigoplus_{(\sigma, \pi)}{\sf RSK}_{\sigma, \pi}^{\oplus N_{\sigma, \pi}(m, n, d)}\right),\]
	where the sum is over all nonzero reduced weight pairs $(\sigma, \pi)$ of degree $d' \leq d$, 
	\[N_{\sigma, \pi}(m, n, d) = \begin{cases}
		A_{\sigma, \pi}(d)\left({m\choose \ell(\sigma)}{n\choose \ell(\pi)}+ {m\choose \ell(\pi)}{n\choose \ell(\sigma)}\right) & \text{if } \sigma\neq \pi,\\
        		A_{\sigma, \pi}(d){m\choose \ell(\sigma)}{n\choose \ell(\pi)} & \text{if } \sigma = \pi,
    		\end{cases}\]
	and
	\[N_0(m, n, d) = {{d+n-1}\choose d}m+{{d+m-1}\choose d}n-mn.\]
\end{theorem}
\begin{proof}
	Given any degree-$d$ weight pair $(\widetilde\sigma, \widetilde\pi)\in\naturals^m\times\naturals^n$, removing all $0$ entries of $\widetilde\sigma$ and $\widetilde\pi$ and then transposing the resulting 
	weight vectors if necessary yields a pair $(\sigma', \pi')$ satisfying Assumption~\ref{ass:basic}. Composing the ``$0$-removal" and ``transposition" maps of Lemma~\ref{lemma:easyRSKcommuting}(II) and (III) 
	gives an RSK-commuting isomorphism proving  ${\sf RSK}_{\widetilde\sigma, \widetilde\pi}\sim{\sf RSK}_{\sigma', \pi'}$. 
	Corollary~\ref{cor:redunique} then identifies a unique reduced pair $(\sigma, \pi)$ such that $R_{\sigma', \pi'}\cong R_{\sigma, \pi}$ 
	via some composition of RSK-commuting variable-division isomorphisms. We claim that $N_{\sigma, \pi}(m, n, d)$ counts the number of $(\widetilde\sigma, \widetilde\pi)$ reduced to $(\sigma, \pi)$ by this procedure. 

	Fix a nonzero reduced pair $(\sigma, \pi)$. By Corollary~\ref{cor:redcount}, there are $A_{\sigma, \pi}(d)$ degree-$d$ pairs $(\sigma', \pi')$ satisfying Assumption~\ref{ass:basic} 
	that reduce to $(\sigma, \pi)$. Each such pair satisfies $\ell(\sigma') = \ell(\sigma)$ and $\ell(\pi') = \ell(\pi)$, so there are exactly ${m\choose\ell(\sigma)}{n\choose\ell(\pi)}$ weight pairs $(\widetilde\sigma,\widetilde\pi)$ 
	such that the $0$-removal map of Lemma~\ref{lemma:easyRSKcommuting}(II) is an RSK-commuting isomorphism $R_{\widetilde\sigma, \widetilde\pi}\to R_{\sigma', \pi'}$. When $\sigma\neq \pi$, we know that 
	$\sigma'\neq\pi'$, so in this case the ${m\choose\ell(\pi)}{n\choose\ell(\sigma)}$ weight pairs equivalent to $(\pi', \sigma')$ via the $0$-removal map of Lemma~\ref{lemma:easyRSKcommuting}(II) are also equivalent to 
	$(\sigma', \pi')$ via the transposition map of Lemma~\ref{lemma:easyRSKcommuting}(III). We have thus identified exactly $N_{\sigma, \pi}(m, n, d)$ weight pairs $(\widetilde\sigma, \widetilde\pi)$ such that 
	${\sf RSK}_{\widetilde\sigma, \widetilde\pi}\sim {\sf RSK}_{\sigma, \pi}$. 

	The isolated $N_0(m, n, d)$ term, associated to the $1\times 1$ identity matrix ${\rm Id}_1$, counts all degree-$d$ pairs $(\widetilde\sigma, \widetilde\pi)$ that reduce all the way to $(\vec{0}, \vec{0})$. 
	Corollary~\ref{cor:redweuse}(I) shows that there is one such pair for each degree-$d$ monomial using only variables from a single row or column of the generic $m\times n$ matrix $Z = [z_{ij}]$, 
	and Proposition~\ref{prop:basicgrowth}(I) and (II) together show that there are no others. The formula for $N_0(m, n, d)$ is a textbook count of these monomials.
\end{proof}

\begin{remark}
	The values $N_{\sigma, \pi}(m, n, d)$ do not depend on the ordering of individual weights in $\sigma$ or $\pi$. One can therefore optimize a bit more by only computing 
	$N_{\sigma, \pi}(m, n, d)$ for reduced weights $(\sigma, \pi)$ where $\sigma$ and $\pi$ are both partitions.
\end{remark}

\begin{example}\label{exa:dupto3}
	We use Theorem~\ref{thm:blockbuild} to describe ${\sf RSK}_{m, n, d}$ for $d\leq 3$. There are four reduced pairs $(\sigma, \pi)$ of degree $0 < d'\leq 3$: 
	we compute the corresponding $g_{\sigma, \pi}$, $A_{\sigma, \pi}(d)$, and $N_{\sigma, \pi}(m, n, d)$ in Table~\ref{table:dupto3} below.

	\begin{table}[h!]
	\centering
	\[
	\begin{array}{|c|c|c|c|l|}
	\hline
	\sigma & \pi  & g_{\sigma, \pi}    &  A_{\sigma, \pi}(d)  & N_{\sigma, \pi}(m, n, d) \\
	\hline
	11 & 11 & 4 & 4(d-2)+\delta_{d, 2} & (4(d-2)+\delta_{d, 2}){m\choose 2}{n\choose 2}\\
	\hline
	21 & 111 & 3 & {{d-1}\choose 2} & {{d-1}\choose 2}\left({m\choose 2}{n\choose 3} + {m\choose 3}{n\choose 2}\right)\\
	\hline
	12 & 111 & 3 & {{d-1}\choose 2} & {{d-1}\choose 2}\left({m\choose 2}{n\choose 3} + {m\choose 3}{n\choose 2}\right)\\
	\hline
	111 & 111 & 0 & \delta_{d,3} & \delta_{d, 3}{m\choose 3}{n\choose 3}\\
	\hline
	\end{array}
	\]
\caption{\label{table:dupto3} Nonzero reduced pairs $(\sigma, \pi)$ of degree $d' \leq 3$}
\end{table}
	All four of the corresponding RSK-matrices appeared previously as examples: ${\sf RSK}_{11, 11}$ is Example~\ref{exa:11-11}, ${\sf RSK}_{111, 111}$ is Example~\ref{exa:111-111}, 
	and ${\sf RSK}_{21,111}$ and ${\sf RSK}_{12, 111}$ are Example~\ref{exa:111/111}. 
	Combining these matrices with Theorem~\ref{thm:blockbuild} completely describes ${\sf RSK}_{m, n, d}$ for $d\leq 3$.
\end{example}

\begin{example}\label{exa:2x2block}
	We use Theorem~\ref{thm:blockbuild} to describe ${\sf RSK}_{2, 2, d}$. Corollary~\ref{cor:redweuse}(II) shows that in this case the only reduced weights $(\sigma, \pi)$ of degree $0<d'\leq d$ are 
	those of the form $\sigma = \pi = (a, a)$ for each $0 < a \leq \lfloor d/2\rfloor$. By Corollary~\ref{cor:redcount}(II) and Theorem~\ref{thm:blockbuild} we have
	\[N_{aa, aa}(2, 2, d) = A_{aa, aa}(d) = 4(d-2a)+\delta_{d, 2a},\quad N_0(2, 2, d) = 4d.\]
	We therefore obtain a relatively simple block decomposition for ${\sf RSK}_{2, 2, d}$:
	\[{\sf RSK}_{2, 2, d} = ({\rm Id}_1^{\oplus 4d})\oplus\left(\bigoplus_{a = 1}^{\lfloor d/2\rfloor}{\sf RSK}_{aa, aa}^{\oplus 4(d-2a)+\delta_{d, 2a}}\right).\]
	Example~\ref{exa:2x2} computes each matrix ${\sf RSK}_{aa, aa}$, making this decomposition fully explicit.
\end{example}

\section{Three families of examples}\label{sec:examples}
Section~\ref{sec:main} presents a detailed reduction from the study of ${\sf RSK}_{m, n, d}$ to ${\sf RSK}_{\sigma, \pi}$, where $(\sigma, \pi)$ is reduced of degree $0 < d'\leq d$. 
In this section we give more explicit descriptions of ${\sf RSK}_{\sigma, \pi}$ for three infinite families of reduced weights $(\sigma, \pi)$. 
These matrices will then be used to establish various results about ${\sf RSK}_{m, n, d}$ in the later sections. 

\subsection{Permutation weights}\label{subsec:famperm}
Let $1^d$ denote the weight vector $(1,\dots, 1)\in\naturals^d$. Then ${\sf Cont}_{1^d, 1^d}$ is the set of all $d\times d$ permutation matrices, 
and ${\sf RSK}_{1^d, 1^d}$ is the matrix describing Schensted insertion as a linear operator. 
Let $\alpha, \beta \in {\sf Mat}_{m,n}({\mathbb N})$ with $(P,Q)={\rm RSK}(\alpha)$. 
Determining ${\sf RSK}_{\sigma, \pi}(\beta, \alpha) = [z^\beta][P|Q]$ is \emph{a priori} difficult; one must determine which terms in the expansion of each minor are factors of $z^\beta$, 
then count (with signs) the number of ways to multiply one such factor from each minor in $[P|Q]$ to obtain $z^\beta$ exactly. 
However, when $\alpha, \beta\in{\sf Cont}_{1^d, 1^d}$ the situation simplifies dramatically. 
Let $\beta_c(\alpha)$ be the submatrix of $\beta$ using row indices from the
$c$-th column of $P$ and column indices from the $c$-th column of $Q$. 
The next proposition allows us to determine individual entries of ${\sf RSK}_{1^d, 1^d}$ quickly, without computing the entire matrix.

\begin{proposition}\label{prop:permweight}
	Let $\alpha, \beta\in{\sf Cont}_{1^d, 1^d}$. Then
	\[{\sf RSK}_{1^d, 1^d}(\beta, \alpha) = \prod_c \det \beta_c(\alpha).\]
\end{proposition}
\begin{proof}
	Let $\alpha\in{\sf Cont}_{1^d, 1^d}$ and $(P, Q) = {\rm RSK}(\alpha)$. Since $\alpha$ is a permutation matrix, $P$ and $Q$ are both standard tableaux, i.e., 
	each entry of $[d]$ appears exactly once in $P$ and once in $Q$. Thus each variable $z_{ij}$ appears in at most one minor of $[P|Q]$.
	For each $\beta\in{\sf Cont}_{1^d, 1^d}$, we know $[z^\beta][P|Q]\neq 0$ if and only if some factor of $z^\beta$ appears in the $c$-th minor of $[P|Q]$. 
	This factor corresponds to a (necessarily unique) nonzero term in $\det\beta_c(\alpha)$, proving the proposition. 
\end{proof}

Proposition~\ref{prop:permweight} shows that all entries of ${\sf RSK}_{1^d, 1^d}$ are products of determinants of partial permutation matrices, hence lie in $\{-1, 0, 1\}$. 
In Section~\ref{sec:trace} we use this description of ${\sf RSK}_{1^d, 1^d}$ to study the trace of ${\sf RSK}_{m, n, d}$. 
Examples~\ref{exa:11-11} and ~\ref{exa:111-111} display ${\sf RSK}_{1^d, 1^d}$ for $d = 2$ and $d=3$ respectively.

\subsection{Voting weights}\label{subsec:fam2x1d}
Our second family of reduced weights $(\sigma, \pi)$ are those where $\pi = 1^d$ and $\sigma$ is unrestricted. 
We call such pairs \emph{voting weights}, thinking of a matrix $\alpha\in{\sf Cont}_{\sigma, 1^d}$ as recording the votes of $d$ individuals for one of $\ell(\sigma)$ candidates, where $\sigma_i$ votes are cast for candidate $i$. 
We immediately obtain a formula for individual entries of ${\sf RSK}_{\sigma, 1^d}$ generalizing that of Proposition~\ref{prop:permweight}.

\begin{proposition}\label{prop:2x1dweight}
	Let $\alpha, \beta\in{\sf Cont}_{\sigma, 1^d}$ with $\ell(\sigma) = 2$. Then
	\[{\sf RSK}_{\sigma, 1^d}(\beta, \alpha) = \prod_c \det \beta_c(\alpha).\]
\end{proposition}
\begin{proof}
	The argument is the same as in Proposition~\ref{prop:permweight}, except that $\alpha, \beta\in {\sf Cont}_{\sigma,1^d}$.
\end{proof}

The following proposition gives a fully explicit description of ${\sf RSK}_{\sigma, 1^d}$ in the special case where $\sigma = (d-1, 1)$. These examples will be used in Section~\ref{sec:eigen}.

\begin{proposition}\label{prop:rouB}
	For an integer $d\geq 2$, ${\sf RSK}_{(d-1, 1), 1^d}$ is equal to the $d\times d$ matrix
	\[A_d:=\begin{bmatrix}
	1 & 1 & 0 & 0 & 0 &  \\
	0 & 0 & 1 & 0 & 0 &  \\
	0 & 0 & 0 & \ddots & 0 &  \\
	0 & 0 & 0 & 0 & 1 & \\
	0 & -1 & -1 & \dots & -1
	\end{bmatrix}.\]
\end{proposition}
\begin{proof}
	Identify each $\alpha\in {\sf Cont}_{(d-1, 1), 1^d}$ by the unique $j$ such that $\alpha_{2, j} = 1$. 
	The canonical contingency table $\alpha^0$ is the one with $\alpha_{2, d} = 1$, and moreover 
	\[{\sf RSK}(z^{\alpha^0})=z^{\alpha^0}=z_{11}z_{12}\cdots z_{1,d-1}z_{2,d}.\] 
	Now suppose $\beta\in{\sf Cont}_{\sigma, \pi}$ is any other table, so
	$\beta_{2, j-1} = 1$ for some $1< j \leq d$.
	Then ${\rm RSK}(\beta)=(P,Q)$ where 
	\ytableausetup{aligntableaux=center, boxsize=1.1em}
	\[P=\begin{ytableau} 
	1 & 1 & \small{\cdots} & 1 \\ 2 \end{ytableau}, 
	Q=\begin{ytableau} 1 & 2 & \small{\cdots} & \small{\cdots} \\ j \end{ytableau},
	\]
	and the first row of $Q$ contains the labels $\{1,2,\ldots,d\}-\{j\}$. Thus 
	\[[P|Q]=\left|
	\begin{matrix}
	z_{11} & z_{1j}\\
	z_{21} & z_{2j}
	\end{matrix}\right|\prod_{k\neq 1,j} z_{1k}.\]
	The form of $A_d$ follows from these considerations.
\end{proof}

\subsection{Triangular weights}\label{subsec:famtri}
Our final family of reduced pairs $(\sigma, \pi)$ are called \emph{triangular} because ${\sf RSK}_{\sigma, \pi}$ turns out to be upper-triangular in the basis ordering of Definition~\ref{def:contorder}.

\begin{definition}
	A reduced weight pair ($\sigma, \pi)$ is \emph{triangular} if $\ell(\sigma) = 2$ and $\sigma_1 = \pi_1$. Such a pair is completely determined by $\pi$, since $\sigma_2 = \sum_{j=2}^{\ell(\pi)} \pi_j$. 
\end{definition}

Describing ${\sf RSK}_{\sigma, \pi}$ for triangular pairs explicitly requires some additional notation.

\begin{definition}\label{def:rescomp}
	Let $\pi\in\naturals^n$. A \emph{$\pi$-bounded composition} is an $n$-tuple $\tau\in\naturals^n$ such that 
	\[\tau_j\leq \pi_j \text{\ for $1\leq j\leq n$ and $|\tau| := \sum_{j=1}^n \tau_j = \pi_1$}.\] 
	Let ${\sf Comp}(\pi)$ denote the set of all $\pi$-bounded compositions, ordered lexicographically. 
\end{definition}

\begin{remark}
	By definition, the elements of ${\sf Comp}(\pi)$ are $\ell(\pi)$-part weak compositions of $\pi_1$, where each part $\tau_j$ is bounded by $\pi_j$. 
	Enumerating these bounded compositions is a textbook problem solved using generating series or inclusion-exclusion (see, e.g., \cite[Sections 6.2, 7.2]{Brualdi}). 
	We are not aware of any closed formula for $|{\sf Comp}(\pi)|$.
\end{remark}

\begin{definition}
	For $\pi\in\naturals^n$, the square matrix $M_\pi$ has rows indexed by $\pi$-bounded compositions. For $\rho, \tau\in {\sf Comp}(\pi)$, 
	the entry of $M_\pi$ in row $\rho$ and column $\tau$ is
	\begin{equation}\label{eqn:trimatrix}
		M_\pi(\rho, \tau) := (-1)^{\pi_1-\rho_1}\prod_{j=2}^n{\tau_j \choose \rho_j}.
	\end{equation}
\end{definition}

\begin{proposition}\label{prop:triangular}
	For any triangular weight pair $(\sigma, \pi)$ and our chosen basis orderings,
	\[{\sf RSK}_{\sigma, \pi} = M_\pi.\] 
\end{proposition}

The first step in proving Proposition~\ref{prop:triangular} is to show that for triangular $(\sigma, \pi)$, ${\sf RSK}_{\sigma, \pi}$ and $M_\pi$ have the same dimensions.
Lemma~\ref{lemma:2rowbasis} below accomplishes this by giving a bijection between ${\sf Cont}_{\sigma, \pi}$ and ${\sf Comp}(\pi)$, showing further that it behaves well with respect to RSK.

\begin{definition}
	Let $T\in {\sf SSYT}(\lambda, n)$. The \emph{$i$th row content vector} of $T$ is 
	\[{\sf row}_i(T) = (c^{(i)}_{i}, c^{(i)}_{i+1}\dots c^{(i)}_n)\in\naturals^{n-i+1},\] where $c^{(i)}_j$ is the number of $j$'s in row $i$ of $T$. 
	Since $c^{(i)}_j = 0$ whenever $j < i$, we only record $c^{(i)}_j$ for $j\geq i$ in ${\sf row}_i(T)$. 
\end{definition}

\begin{lemma}\label{lemma:2rowbasis}
	Let $(\sigma, \pi)$ denote a triangular weight pair and let $\alpha\in{\sf Cont}_{\sigma, \pi}$. Then
	\begin{itemize}
		\item[(I)] The map 
		\[\alpha\mapsto\tau_\alpha:=(\alpha_{1, 1},\dots,\alpha_{1, \ell(\pi)})\] 
		is a bijection between ${\sf Cont}_{\sigma, \pi}$ and ${\sf Comp}(\pi)$. Moreover, this bijection preserves the orderings on ${\sf Cont}_{\sigma, \pi}$ and ${\sf Comp}(\pi)$
		 given in Definitions~\ref{def:contorder} and \ref{def:rescomp} respectively. 
		\item[(II)] If ${\rm RSK}(\alpha) = (P,Q)$, then 
		\[{\sf row}_2(Q) = (\alpha_{1, 2}, \alpha_{1, 3},\dots, \alpha_{1, \ell(\pi)}).\]
	\end{itemize}
\end{lemma}
\begin{proof}
	(I) The given function $\alpha\mapsto\tau_\alpha$ maps ${\sf Cont}_{\sigma, \pi}$ into $\naturals^{\ell(\pi)}$. By assumption, 
	\[\alpha_{1, j}\leq \pi_j  \text{ \ for $2\leq j\leq \ell(\pi)$},\] 
	and by triangularity we also have
	\[|\tau_\alpha| = \sigma_1 = \pi_1.\] 
	Thus $\tau_\alpha$ lies in ${\sf Comp}(\pi)$. 
	Now let $\tau\in {\sf Comp}(\pi)$ be arbitrary and begin defining a matrix 
	\[\alpha^{\tau}\in {\sf Cont}_{\sigma, \pi}\] 
	by setting its first row equal to $\tau$. By triangularity, the first row of $\alpha^{\tau}$ then sums to $\sigma_1$, and by the definition of ${\sf Comp}(\pi)$ 
	\[\alpha^{\tau}_{1, j} \leq \pi_j \text{ for all $1\leq j\leq \ell(\pi)$.}\] 
	Thus we may (and must) set
	\[\alpha^{\tau}_{2, j} := \pi_j - \tau_j,\]
	constructing a unique matrix associated to $\tau$ in ${\sf Cont}_{\sigma, \pi}$ and establishing the bijection. 

	To prove our bijection is order-preserving, note that if $\alpha, \beta\in{\sf Cont}_{\sigma, \pi}$ with $\ell(\sigma) = 2$ then $\alpha_{2, j} \neq \beta_{2, j}$ if and only if 
	$\alpha_{1, j} \neq \beta_{1, j}$. Thus the first position in which $\alpha$ and $\beta$ differ (in the reading order of Definition~\ref{def:contorder}) must occur in the first row, 
	implying that $\alpha > \beta$ if and only if $\tau_\alpha > \tau_\beta$.

	(II)  It suffices to show that for any $\alpha\in {\sf Cont}_{\sigma, \pi}$, every biletter $(1|j)$ with $j\geq 2$ arising from $\alpha$ lies in a bump chain of length $2$. 
	This follows from the triangularity of $(\sigma, \pi)$:  we know that 
	\[\alpha_{2, 1} = \pi_1-\alpha_{1, 1} = \sigma_1 - \alpha_{1, 1} =\sum_{j=2}^{\ell(\pi)} \alpha_{1, j},\] 
	so each $(1|j)$ with $j\geq 2$ bumps a biletter $(2|1)$ during the insertion algorithm.
\end{proof}

\begin{proof}[Proof of Proposition~\ref{prop:triangular}]
	Let $\alpha\in{\sf Cont}_{\sigma, \pi}$, and let $\tau \in {\sf Comp}(\pi)$ be the $\pi$-bounded composition corresponding to $\alpha$ via Lemma~\ref{lemma:2rowbasis}(I). Let $(P, Q) := {\rm RSK}(\alpha)$. 
	Then $(P, Q)$ is determined by ${\sf row}_2(Q)$, which is $(\tau_2,\dots, \tau_{\ell(\pi)})$ by Lemma~\ref{lemma:2rowbasis}(II). 
	Since $\sigma_1 = \pi_1$ by assumption, $[P|Q]$ contains exactly $\tau_j$ columns of the form 
	\[\left[\ \begin{ytableau} 1\\ 2\end{ytableau} \  \bigg\vert\  \begin{ytableau}1\\ j\end{ytableau} \ \right] 
	= \begin{vmatrix}z_{11} & z_{1j}\\ z_{21} & z_{2j}\end{vmatrix} = z_{11}z_{2j}-z_{21}z_{1j}, \ (2\leq j\leq \ell(\pi)),\]
	and some number of length-$1$ columns, each of the form 
	\[\setlength{\delimitershortfall}{-1pt}\left[\ \begin{ytableau} 1 \end{ytableau} \  \vert\  \begin{ytableau} 1 \end{ytableau} \ \right] = z_{11} 
	\text{ or } \left[\ \begin{ytableau} 2 \end{ytableau} \  \vert\  \begin{ytableau} j \end{ytableau} \ \right] = z_{2j}, \ (2\leq j\leq \ell(\pi)).\]
	A monomial in the expansion of $[P|Q]$ is determined by a choice of term from each length-$2$ column while performing the multiplication. 
	Now, let $z^\beta\in R_{\sigma, \pi}$ be a monomial term in the expansion of $[P|Q]$, determined by the vector of entries 
	\[\rho := (\beta_{1, 2},\dots, \beta_{1, n}).\] 
	From the explicit description of the columns of $[P|Q]$ above we see that $\rho_j = \beta_{1, j}$ is the number of $z_{21}z_{1j}$ terms chosen while expanding the product. 
	Thus the monomial $z^\beta$ appears $\prod_{j=2}^{\ell(\pi)} {\tau_j \choose \rho_j}$ times in the expanded product $[P|Q]$. 
	Each time, the sign on $z^\beta$ is $(-1)^r$ where $r = \sum_{j=2}^{\ell(\pi)} \rho_j = \pi_1-\rho_1$. It follows that 
	\[{\sf RSK}_{\sigma, \pi}(\rho, \tau) = (-1)^{\pi_1-\rho_1}\prod_{j=2}^{\ell(\pi)}{\tau_j\choose \rho_j} = M_\pi(\rho, \tau),\]
	as claimed.
\end{proof}

\begin{example}
	Let 
	\[(\sigma, \pi) = (35, 323).\] 
	Then $(\sigma, \pi)$ is triangular and $\dim R_{\sigma, \pi} = 9$. 
	We write the (ordered) basis twice: first with matrices in ${\sf Cont}_{\sigma, \pi}$, then with the corresponding $Q$-tableaux. 
	\[\begin{bmatrix} 3 & 0 & 0\\ 0 & 2 & 3\end{bmatrix},\quad \begin{bmatrix} 2 & 1 & 0\\ 1 & 1 & 3\end{bmatrix},\quad \begin{bmatrix} 2 & 0 & 1\\ 1 & 2 & 2\end{bmatrix},\quad 
		\begin{bmatrix} 1 & 2 & 0\\ 2 & 0 & 3\end{bmatrix},\quad \begin{bmatrix} 1 & 1 & 1\\ 2 & 1 & 2\end{bmatrix},\quad \begin{bmatrix} 1 & 0 & 2\\ 2 & 2 & 1 \end{bmatrix},\]
	\[\begin{bmatrix} 0 & 2 & 1\\ 3 & 0 & 2\end{bmatrix},\quad \begin{bmatrix} 0 & 1 & 2\\ 3 & 1 & 1\end{bmatrix},\quad \begin{bmatrix} 0 & 0 & 3\\ 3 & 2 & 0\end{bmatrix};\]
	\[\begin{ytableau}1 & 1 & 1 & 2 & 2 & 3 & 3 & 3\end{ytableau},\quad \begin{ytableau}1 & 1 & 1 & 2 & 3 & 3 & 3\\ 2\end{ytableau},\quad \begin{ytableau}1 & 1 & 1 & 2 & 2 & 3 & 3\\ 3\end{ytableau},\]
	\[\begin{ytableau}1 & 1 & 1 & 3 & 3 & 3\\ 2 & 2\end{ytableau},\quad \begin{ytableau}1 & 1 & 1 & 2 & 3 & 3\\ 2 & 3\end{ytableau},\quad \begin{ytableau}1 & 1 & 1 & 2 & 2 & 3\\ 3 & 3\end{ytableau},\]
	\[\begin{ytableau}1 & 1 & 1 & 3 & 3\\ 2 & 2 & 3\end{ytableau},\quad \begin{ytableau}1 & 1 & 1 & 2 & 3\\ 2 & 3 & 3\end{ytableau},\quad \begin{ytableau}1 & 1 & 1 & 2 & 2\\ 3 & 3 & 3\end{ytableau}.\]
	By Proposition~\ref{prop:triangular} or explicit computation we then see that
	\[{\sf RSK}_{\sigma, \pi} =
	\begin{bmatrix} 
	1 & 1 & 1 & 1 & 1 & 1 & 1 & 1 & 1\\
	0 & -1 & 0 & -2 & -1 & 0 & -2 & -1 & 0\\
	0 & 0 & -1 & 0 & -1 & -2 & -1 & -2 & -3\\
	0 & 0 & 0 & 1 & 0 & 0 & 1 & 0 & 0\\
	0 & 0 & 0 & 0 & 1 & 0 & 2 & 2 & 0\\
	0 & 0 & 0 & 0 & 0 & 1 & 0 & 1 & 3\\
	0 & 0 & 0 & 0 & 0 & 0 & -1 & 0 & 0\\
	0 & 0 & 0 & 0 & 0 & 0 & 0 & -1 & 0\\
	0 & 0 & 0 & 0 & 0 & 0 & 0 & 0 & -1\\
	\end{bmatrix}.\]
\end{example}

\begin{example}\label{exa:2x2}
	Suppose $\ell(\sigma) = \ell(\pi) = 2$. 
	By reduction via Corollary~\ref{cor:redweuse}(II), we may assume without loss of generality that for some $a\in\naturals$, 
	\[(\sigma, \pi) = (aa, aa).\] 
	Then $(\sigma, \pi)$ is triangular, so by Proposition~\ref{prop:triangular} we know ${\sf RSK}_{\sigma, \pi} = M_{aa}$. 
	The rows and columns of $M_{aa}$ are indexed by $2$-tuples 
	$\tau = (\tau_1, \tau_2)\in\naturals^2$ 
	such that 
	\[\tau_1+\tau_2 \leq a  \text{ \ and $\tau_1, \tau_2\leq a$}.\] 
	These pairs correspond to nonnegative integers $k\leq a$ (explicitly, $k\leftrightarrow(a-k,k)$). 
	The formula (\ref{eqn:trimatrix}) for $M_{aa}$ then simplifies to
	\[M_{aa}(k, \ell) = (-1)^k{\ell\choose k}.\]
\end{example}

\begin{corollary}\label{cor:permbij}
	If $(\sigma, \pi)$ is triangular and $\widetilde\pi$ is obtained from $\pi$ by permuting the $n-1$ entries $(\pi_2,\dots, \pi_n)$ by some permutation $w$, then ${\sf RSK}_{\sigma, \pi}\sim {\sf RSK}_{\sigma, \widetilde\pi}$. 
\end{corollary}
\begin{proof}
	The permutation map $w$ sending $\pi$ to $\widetilde\pi$ induces a non-order-preserving bijection 
	\[{\sf Comp}(\pi)\xrightarrow{w}{\sf Comp}(\widetilde\pi)\] given by 
	mapping $(\tau_1, \tau_2,\dots, \tau_n)$ to $w(\tau) := (\tau_1, \tau_{w(2)},\dots, \tau_{w(n)})$. 
	From the defining equation (\ref{eqn:trimatrix}) we then see that
	\[M_{\widetilde\pi}(w(\rho), w(\tau)) = (-1)^{\pi_1-\rho_1}\prod_{j=2}^n{\tau_{w(j)} \choose \rho_{w(j)}} = (-1)^{\pi_1-\rho_1}\prod_{j=2}^n{\tau_j \choose \rho_j} = M_{\pi}(\rho, \tau). \]
	Thus 
	$M_\pi\sim M_{\widetilde\pi}$, 
	with the change of basis induced by the bijection 
	${\sf Comp}(\pi)\xrightarrow{w}{\sf Comp}(\widetilde\pi)$ 
	on the basis vectors. 
	It follows from Proposition~\ref{prop:triangular} that 
	\[{\sf RSK}_{\sigma, \pi} = M_\pi \sim M_{\widetilde\pi} = {\sf RSK}_{\sigma, \widetilde\pi}.\qedhere\]
\end{proof}

\section{Eigenvalues}\label{sec:eigen}
With our families of examples in hand, we apply Theorem~\ref{thm:suffstab} and Theorem~\ref{thm:blockbuild} to study the properties of the matrix ${\sf RSK}$, starting with its eigenvalues. 
Our first observation is that all roots of unity occur infinitely often as eigenvalues of ${\sf RSK}$.

\begin{theorem}\label{thm:rootsofunity}
If $m\geq 2$ and $n,d\geq k$, then all $k$-th roots of unity are eigenvalues for ${\sf RSK}_{m, n, d}$. 
\end{theorem}

We prove Theorem~\ref{thm:rootsofunity} using the voting weights $(\sigma, \pi) = ((d-1, 1), 1^d)$ of Section~\ref{subsec:fam2x1d}.

\begin{lemma}\label{lemma:rouC}
The characteristic polynomial of the matrix $A_d$ in Proposition~\ref{prop:rouB} is $p_{A_d}(t)=t^d-1$.
\end{lemma}
\begin{proof}
By cofactor expansion of $p_{A_d}(t)=t {\rm Id}_d - A_d$ along the first column, 
\[p_{A_d}(t)=(t-1)p_{B_d}(t)\] 
where $B_d$ is the southeast $(d-1)\times (d-1)$ submatrix of $A_d$. However, $B_d$ is precisely the
companion matrix of $1+t+t^2+\cdots+t^{d-1}$. The result follows.
\end{proof}

\begin{proof}[Proof of Theorem~\ref{thm:rootsofunity}:]
For any $m\geq 2$ and $n,d \geq k$, Theorem~\ref{thm:blockbuild} and Corollary~\ref{cor:redweuse}(III) show that 
${\sf RSK}_{m, n, d}$ contains at least one block similar to ${\sf RSK}_{\sigma, \pi}$ for $\sigma = (k-1, 1)$ and $\pi = 1^k$. 
Proposition~\ref{prop:rouB} proves that
${\sf RSK}_{\sigma,\pi} = A_k$. 
Now apply Lemma~\ref{lemma:rouC}.
\end{proof}

On the other hand, very few rational numbers occur as eigenvalues of ${\sf RSK}_{m, n, d}$.

\begin{proposition}\label{prop:pm1}
Let $(\sigma, \pi)$ be a reduced pair. 
\begin{itemize}
\item[(I)] All rational eigenvalues of ${\sf RSK}_{\sigma,\pi}$ are $\pm1$.
\item[(II)] ${\sf RSK_{\sigma, \pi}}$ always has $1$ as an eigenvalue. 
\end{itemize}
\end{proposition}
\begin{proof} (I): By Proposition~\ref{prop:det}, 
$\det {\sf RSK}_{\sigma, \pi}\in \{\pm 1\}$. 
Thus, the constant
term of the characteristic polynomial $p_{{\sf RSK}_{\sigma, \pi}}(t)$ is $1$. Now, every entry of ${\sf RSK}_{\sigma, \pi}$ is integral, and hence
\[p_{{\sf RSK}_{\sigma, \pi}}(t)\in {\mathbb Z}[t].\] 
The claim then follows from the rational root theorem.

(II): The monomial $z^{\alpha^0}$, for $\alpha^0\in{\sf Cont}_{\sigma, \pi}$ as in Definition~\ref{def:cantable}, is fixed under ${\sf RSK}$.
\end{proof}

\begin{remark}
The proof of Proposition~\ref{prop:pm1}(II) shows that the canonical table $\alpha^0\in{\sf Cont}_{\sigma, \pi}$ always indexes a $1$-eigenvector of ${\sf RSK}_{\sigma, \pi}$. Other $1$-eigenvectors are possible; for example,  
\begin{align*}
{\sf RSK}(z_{12}^2 z_{21}^2 - z_{11}z_{12}z_{21}z_{22}) & =\left[\ \begin{ytableau}1 & 1\\ 2 & 2\end{ytableau} \ \bigg\vert\ \begin{ytableau}1 & 1\\ 2 & 2\end{ytableau} \ \right]
-\left[\ \begin{ytableau}1 & 1 & 2\\ 2 \end{ytableau}\ \bigg\vert\ \begin{ytableau}1 & 1 & 2 \\ 2\end{ytableau}  \ \right]\\
& = 
\left|
\begin{matrix}
z_{11} & z_{12} \\ 
z_{21} & z_{22}
\end{matrix}
\right|^2
-
\left|
\begin{matrix}
z_{11} & z_{12} \\ 
z_{21} & z_{22}
\end{matrix}
\right| z_{11}z_{22}\\
& = z_{12}^2 z_{21}^2 - z_{11}z_{12}z_{21}z_{22}.
\end{align*}
\end{remark}

\begin{conjecture}\label{conj:complex}
If $\ell(\sigma),\ell(\pi)\geq 3$, then ${\sf RSK}_{\sigma,\pi}$ has a non-real eigenvalue.
\end{conjecture}

Conjecture~\ref{conj:complex} has been checked for $\sigma,\pi\in {\mathbb N}^3$ of degree $d\leq 8$. 
For a family of matrices ${\sf RSK}_{\sigma, \pi}$ with only integer eigenvalues, we turn to the triangular weights of Section~\ref{subsec:famtri}.

\begin{proposition}\label{thm:trieigen}
	If $(\sigma, \pi)$ is triangular, then ${\sf RSK}_{\sigma, \pi}$ is an upper-triangular matrix. Its set of eigenvalues is $\{1, -1\}$, with respective multiplicities 
	$|\{\rho\in{\sf Comp}(\pi): \pi_1-\rho_1 \text{ is even}\}|$ and $|\{\rho\in{\sf Comp}(\pi): \pi_1-\rho_1\text{ is odd}\}|$.
\end{proposition}
\begin{proof}
	By Proposition~\ref{prop:triangular}, it suffices to consider the matrix $M_\pi$. We first show that $M_\pi$ is upper-triangular. 
	If $\tau > \rho$ in the lexicographic order on ${\sf Comp}(\pi)$ then 
	$\tau_j < \rho_j$ 
	for some $2\leq j\leq \ell(\pi)$, so 
	\[M_\pi(\rho, \tau) = 0\] 
	by the defining formula (\ref{eqn:trimatrix}) for $M_\pi$. 
	Therefore its eigenvalues are simply the diagonal entries $M_\pi(\rho, \rho)$. From (\ref{eqn:trimatrix}) again we compute 
	\[M_\pi(\rho, \rho) = 
	\begin{cases} 
	1 & \text{if } \pi_1-\rho_1 \text{ is even}, \\
	-1 & \text{if } \pi_1-\rho_1 \text{ is odd},
	\end{cases}\]
	so the eigenvalues of $M_\pi$ have the claimed multiplicities. 
\end{proof}

\begin{example}\label{exa:2x2eigen}
	Suppose $\ell(\sigma) = \ell(\pi) = 2$, so $(\sigma, \pi) = (aa, aa)$ as in Example~\ref{exa:2x2}. Indexing rows and columns starting from $0$, that example showed 
	\[{\sf RSK}_{aa, aa}(k, \ell) = (-1)^k{\ell\choose k}.\]
	Reading off the diagonal entries of $M_{aa}$ shows that ${\sf RSK}_{aa, aa}$ has eigenvalues $1$ and $-1$ with multiplicities $\lceil\frac{a+1}{2}\rceil$ and $\lfloor\frac{a+1}{2}\rfloor$ respectively, 
	in accordance with Proposition~\ref{thm:trieigen}.
\end{example}

\begin{conjecture}\label{conj:nontri}
	If a reduced pair $(\sigma, \pi)$ is not triangular, then ${\sf RSK}_{\sigma, \pi}$ has a non-integer eigenvalue.
\end{conjecture}

\begin{remark}
	Conjecture~\ref{conj:nontri} is equivalent to the claim that if $(\sigma, \pi)$ is not triangular, then 
	there is no \emph{unimodular} matrix $\psi$ such that $\psi\cdot {\sf RSK}_{\sigma, \pi}\cdot \psi^{-1}$ is upper-triangular \cite[Theorem 2]{Sidorov}. 
	With Proposition~\ref{thm:trieigen}, Conjecture~\ref{conj:nontri} would fully characterize pairs $(\sigma, \pi)$ such that ${\sf RSK}_{\sigma, \pi}$ has all integer eigenvalues. 
	We checked Conjecture~\ref{conj:nontri} for all reduced weights $(\sigma, \pi)$ with $(\ell(\sigma), \ell(\pi)) = (2, 3)$ and degree at most $9$, 
	as well as those with $(\ell(\sigma), \ell(\pi)) = (3, 3)$ and degree at most $8$.
\end{remark}

The next result shows that one cannot expect a fully explicit description of the eigenvalues of ${\sf RSK}_{m, n, d}$. We pose the problem of instead giving additional necessary or sufficient conditions for a polynomial 
$f(t)\in \integers[t]$ to be the characteristic polynomial of some ${\sf RSK}_{\sigma, \pi}$. 

\begin{theorem}\label{thm:radicals}
	If $m, n\geq 3$ and $d\geq 4$, then the characteristic polynomial of ${\sf RSK}_{m, n, d}$ is not solvable by radicals.
\end{theorem}
\begin{proof}
	Let 
	$(\sigma,\pi)=(211, 121)$. 
	Then
	\[{\sf RSK}_{211,121}=
	\begin{bmatrix}
	1 & 1 & 0 & 1 & 1 & 0 & 1\\
	0 & 0 & 1 & 0 & 0 & 1 & -1\\
	0 & 0 & 0 & 0 & -1 & 0 & 0\\
	0 & 0 & 0 & -1 & -1 & 0 & -1\\
	0 & 0 & -1 & 0 & 1 & 0 & 1\\
	0 & 0 & 0 & 0 & 0 & -1 & 1\\
	0 & -1 & 0 & 0 & 0 & 0 & -1
	\end{bmatrix}.\]
	The characteristic polynomial of this matrix is 
	\[p_{{\sf RSK}_{211, 121}}(t)=(t-1)(t+1)(t^5+t^4-3t^3-2t^2-t-1).\]
	The quintic factor $f(t) = t^5+t^4-3t^3-2t^2-t-1$ is irreducible over ${\mathbb Z}$ (and ${\mathbb Q}$):
	reducing $f(t)$ modulo $2$ gives
	$t^5+t^4+t^3+t+1$, which is irreducible over ${\mathbb Z}_2$. Using basic calculus methods,
	one proves $f(t)$ has exactly three real roots ($\approx -2.05, -0.76, 1.72$) and thus two complex roots ($\approx 0.04\pm 0.61i$).
	This implies the Galois group over ${\mathbb Q}$ is the full symmetric group ${\mathfrak S}_5$, which is not solvable. Hence the roots
	of $f(t)$ are not solvable by radicals. Since $A_{211, 121}(d)=1$ for $d\geq 4$ by Corollary~\ref{cor:redcount}, 
	Theorem~\ref{thm:blockbuild} shows that these eigenvalues occur for all ${\sf RSK}_{m,n, d}$ with $m,n\geq 3$ and $d\geq 4$.
\end{proof}

\section{Diagonalizability}\label{sec:diag}
The main theorem of this section, Theorem~\ref{thm:maindiag}, classifies triples $(m, n, d)$ such that ${\sf RSK}_{m, n, d}$ is diagonalizable. It is a consequence of Theorem~\ref{thm:blockbuild} and explicit calculations. 
The next proposition summarizes a family of these calculations. 

\begin{proposition}\label{prop:tridiag}
	If $(\sigma, \pi)$ is triangular then ${\sf RSK}_{\sigma, \pi}$ is diagonalizable.
\end{proposition}
\begin{proof}
	By Proposition~\ref{prop:triangular}, it suffices to consider the matrix $M_\pi$. We will show that the minimal polynomial of $M_\pi$ splits as
	$\mu_{M_{\pi}}(t)=(t+ 1)(t-1)$. 
	We check the equivalent condition that $M_\pi$ squares to the identity matrix:
	\begin{align*}
		M_\pi^2(\rho, \tau) &= \sum_{\omega}M_\pi(\rho, \omega)\cdot M_\pi(\omega, \tau)\\
		&= \sum_{\omega = \rho}^{\tau}\left((-1)^{(\pi_1-\rho_1)+(\pi_1-\omega_1)}\prod_{j=2}^{\ell(\pi)}{\omega_j\choose \rho_j}{\tau_j\choose \omega_j}\right)
		\end{align*}
		\begin{align*}
				&= \sum_{\omega = \rho}^{\tau}\left((-1)^{(\pi_1-\rho_1)+(\pi_1-\omega_1)}\prod_{j=2}^{\ell(\pi)}{\tau_j\choose \rho_j}{\tau_j-\rho_j\choose \omega_j-\rho_j}\right)\\
					&= \left((-1)^{\pi_1-\rho_1}\prod_{j=2}^{\ell(\pi)}{\tau_j\choose \rho_j}\right)\left(\sum_{\omega = \rho}^{\tau}\left(\prod_{j=2}^{\ell(\pi)}(-1)^{\omega_j}{\tau_j-\rho_j\choose \omega_j-\rho_j}\right)\right)\\
							&= \left((-1)^{\pi_1-\rho_1}\prod_{j=2}^{\ell(\pi)}{\tau_j\choose \rho_j}\right)\left(\prod_{j=2}^{\ell(\pi)}\left(\sum_{\omega_j = \rho_j}^{\tau_j}(-1)^{\omega_j}{\tau_j-\rho_j\choose \omega_j-\rho_j}\right)\right)\\
		&= \begin{cases} 0 & \text{if }\rho\neq\tau,\\ (-1)^{(\pi_1-\rho_1)+(\pi_1-\tau_1)} = 1 & \text{if }\rho = \tau.\end{cases}\qedhere
	\end{align*}
\end{proof}

With Proposition~\ref{prop:tridiag} proved, we are ready to characterize the diagonalizable ${\sf RSK}_{m, n, d}$. 
Since ${\sf RSK}_{m, n, d} \sim {\sf RSK}_{n, m, d}$ by Lemma~\ref{lemma:easyRSKcommuting}(III) and ${\sf RSK}_{1, n, d}$ is the identity matrix (hence diagonalizable), we may assume that $2\leq m\leq n$. Theorem~\ref{thm:maindiag} below gives a different characterization of diagonalizability than
Theorem~\ref{thm:sampler}(I). Afterwards we give a proof of their equivalence. 

\begin{theorem}\label{thm:maindiag} Let $2\leq m\leq n$.
	The matrix ${\sf RSK}_{m, n, d}$ is diagonalizable if and only if one of the following mutually exclusive conditions holds:
	\begin{itemize}
		\item[(I)] $m=n=2$,
		\item[(II)] $m=2$, $n=3$, and $d\leq 6$,
		\item[(III)] $m>2$ or $n>3$, and $d\leq 3$.
	\end{itemize}
\end{theorem}
\begin{proof}
	(I) Suppose that $m=n= 2$. 
	Then Theorem~\ref{thm:blockbuild} implies 
	that each block in ${\sf RSK}_{2, 2, d}$ is either the $1\times 1$ identity matrix or a matrix ${\sf RSK}_{aa, aa}$ for some $1\leq a\leq \lfloor d/2\rfloor$ (see Example~\ref{exa:2x2block}). 
	Since each pair $(aa, aa)$ is triangular, ${\sf RSK}_{2, 2, d}$ is diagonalizable for all $d$ by Proposition~\ref{prop:tridiag}. 

	(II) Suppose that $m = 2$ and $n = 3$. Consider the reduced pair $(43, 223)$ of degree $7$. Corollary~\ref{cor:redcount} shows that $A_{43, 223} = {d-7\choose 0}$, so
	by Theorem~\ref{thm:blockbuild} the matrix ${\sf RSK}_{2, 3, d}$ contains a copy of ${\sf RSK}_{43, 223}$ if and only if $d>6$. Direct computation shows ${\sf RSK}_{43, 223}$ is not 
	diagonalizable. Indeed, the matrix is 
	\[{\sf RSK}_{43, 223} = \begin{bmatrix}
	1 & 1 & 1 & 0 & 0 & 1 & 0 & 0\\
	0 & 0 & 0 & 1 & 1 & -1 & 0 & 1\\
	0 & 0 & 0 & 0 & 0 & 0 & 1 & -1\\
	0 & -1 & -2 & -1 & -1 & -2 & 0 & -1\\
	0 & 0 & 0 & 0 & -1 & 2 & -2 & 0\\
	0 & 0 & 0 & 0 & 0 & 0 & 0 & 1\\
	0 & 0 & 1 & 0 & 1 & 1 & 1 & 1\\
	0 & 0 & 0 & 0 & 0 & -1 & 0 & -1 
	\end{bmatrix},\]
	with characteristic and minimal polynomials
	\[p_{{\sf RSK}_{43, 223}}(t) = (t-1)^2(t^2+t+1)^3\text{ and }\mu_{{\sf RSK}_{43, 223}}(t) = (t-1)^2(t^2+t+1)^2.\]
	Thus ${\sf RSK}_{2, 3, d}$ is not diagonalizable for $d>6$. A finite computation then confirms that ${\sf RSK}_{\sigma, \pi}$ is diagonalizable for all reduced pairs $(\sigma, \pi)$ 
	contributing to ${\sf RSK}_{2, 3, d}$ for $d\leq 6$. 

	(III) First, suppose that $n>3$ and consider the reduced pairs $(22, 1111)$ and $(32, 2111)$ of degree $4$ and $5$ respectively. Corollary~\ref{cor:redcount} shows that $A_{22, 1111} = \delta_{d, 4}$
	and $A_{32, 2111} = {d-5\choose 0}$, so by Theorem~\ref{thm:blockbuild} the matrix ${\sf RSK}_{m, n, d}$ contains a copy of ${\sf RSK}_{22, 1111}$ or ${\sf RSK}_{32, 2111}$ whenever $n > 3$ and $d > 3$. 
	Direct computation shows that ${\sf RSK}_{22, 1111}$ and ${\sf RSK}_{32, 2111}$ are not diagonalizable. Explicitly, the matrices are 
	\[{\sf RSK}_{22, 1111} = \begin{bmatrix}
	1 & 1 & 1 & 0 & 0 & 1\\
	0 & 0 & 0 & 1 & 1 & 0\\
	0 & 0 & 0 & 0 & -1 & -1\\
	0 & -1 & 0 & -1 & -1 & -1\\
	0 & 0 & -1 & 0 & 1 & 0\\
	0 & 0 & 0 & 0 & 0 & 1
	\end{bmatrix} \text{ and } {\sf RSK}_{32, 2111} = \begin{bmatrix}
	1 & 1 & 1 & 0 & 0 & 1 & 0\\
	0 & 0 & 0 & 1 & 1 & 0 & 0\\
	0 & 0 & 0 & 0 & 0 & 0 & 1\\
	0 & -1 & 0 & -1 & -1 & -1 & 0\\
	0 & 0 & -1 & 0 & 0 & -1 & -1\\
	0 & 0 & 0 & 0 & -1 & 0 & -1\\
	0 & 0 & 0 & 0 & 1 & 1 & 1
	\end{bmatrix},
	\]
	with characteristic and minimal polynomials 
	\begin{align*}
		p_{{\sf RSK}_{22, 1111}}(t) = \mu_{{\sf RSK}_{21, 1111}}(t) &= (t-1)^2(t^2-t-1)(t^2+t+1),\\
		p_{{\sf RSK}_{32, 2111}}(t) = \mu_{{\sf RSK}_{32, 2111}}(t) & = (t-1)^2(t^2+t+1)(t^3+t+1).
	\end{align*}
	Thus ${\sf RSK}_{m, n, d}$ is not diagonalizable for $n>3$ and $d>3$. 

	Now suppose that $m>2$ and consider the reduced pair $(211, 211)$ of degree $4$. By Corollary~\ref{cor:redcount} $A_{211, 211} = {d-4\choose 0}$, so by Theorem~\ref{thm:blockbuild} the matrix 
	${\sf RSK}_{m, n, d}$ contains a copy of ${\sf RSK}_{211, 211}$ whenever $n>2$ and $d>3$. Another direct computation shows that ${\sf RSK}_{211, 211}$ is not diagonalizable: the matrix is
	\[{\sf RSK}_{211,211}=\begin{bmatrix}
	1 & 1 & 1 & 0 & 0 &  1 & 1\\
	0 & 0 & 0 & 1 & 1 & -1 & 0\\
	0 & 0 & -1 & 0 & 0 & -1 & -1\\
	0 & 0 & 0 & -1 & 0 & 1 &0\\
	0 & 0 & 0 & 0 & -1 & 1 &0\\
	0 & -1 & 0 & 0 & 0 & -1 &-1\\
	0 &0 &0 &0&0&0 &1
    	\end{bmatrix},\]
	with characteristic and minimal polynomials 
	\[p_{{\sf RSK}_{211,211}}(t)=(t-1)^2(t+1)^2(t^3+2t^2+1) \text{\ and \ } \mu_{{\sf RSK}_{211,211}}(t)=(t-1)^2(t+1)(t^3+2t^2+1).\]		
	Thus ${\sf RSK}_{m, n, d}$ is not diagonalizable for $m > 2$ and $d > 3$.

	Finally, we confirm that ${\sf RSK}_{m, n, d}$ is diagonalizable for all $m$ and $n$ when $d\leq 3$. By Theorem~\ref{thm:blockbuild}, the matrix ${\sf RSK}_{m, n, d}$ is a direct sum of copies of 
	the $1\times 1$ identity matrix, ${\sf RSK}_{11, 11}$, ${\sf RSK}_{21, 111}$, ${\sf RSK}_{12, 111}$, and ${\sf RSK}_{111, 111}$ (see Example~\ref{exa:dupto3}). 
	Explicit computation shows that all four matrices are diagonalizable. Indeed, ${\sf RSK}_{111, 111}$ is the only one in which eigenvalues occur with multiplicity 
	($-1$ occurs with multiplicity $2$), and Example~\ref{exa:111-111} provides an explicit basis of eigenvectors for this eigenspace. 
\end{proof}

\begin{proof}[Proof of Theorem~\ref{thm:sampler}(I)]
	We need to show that the diagonalizabity characterization from Theorem~\ref{thm:sampler}(I) agrees with that of Theorem~\ref{thm:maindiag}. 
	By Lemma~\ref{lemma:easyRSKcommuting}(III), we may assume that $m\leq n$. It is clear from the statement of Theorem~\ref{thm:maindiag} that ${\sf RSK}_{m, n, d}$ is diagonalizable whenever $d\leq 3$. Thus we may assume $d>3$.
	If $\mathcal{G}_{m, n, d}$ is of Dynkin type $A$, then $m\leq 1$, so as observed in the paragraph before the proof of Theorem~\ref{thm:maindiag} it follows that ${\sf RSK}_{m, n, d}$ is diagonalizable. 
	If $\mathcal{G}_{m, n, d}$ is of Dynkin type $D$, then $m = n = 2$ and thus ${\sf RSK}_{m, n, d}$ is diagonalizable by Theorem~\ref{thm:maindiag}. 
	If $\mathcal{G}_{m, n, d}$ is of Dynkin type $E$, then $m = 2$ and $n = 3$ and thus ${\sf RSK}_{m, n, d}$ is diagonalizable if and only if $d\leq 6$ by Theorem~\ref{thm:maindiag}. 
	The cases $d = 3, 4, 5, 6$ correspond to the diagrams $E_6, E_7, E_8, E_9$ in Figure~\ref{fig:dynkin}. 
	If $\mathcal{G}_{m, n, d}$ is not of one of the above types, then either $n > 3$ or $m > 2$. Thus ${\sf RSK}_{m, n, d}$ is not diagonalizable by Theorem~\ref{thm:maindiag}, completing the proof.
\end{proof}

In general, we do not know when an individual block ${\sf RSK}_{\sigma, \pi}$ is diagonalizable. 

\begin{remark}\label{remark:noncombo-ex}
	When ${\sf RSK}_{\sigma, \pi}$ and ${\sf RSK}_{\widetilde\sigma, \widetilde\pi}$ are both diagonalizable,
	the two matrices are similar if and only if they have the same eigenvalues (counted with multiplicity).  
	This reveals pairs of similar matrices where the RSK-commuting isomorphism $R_{\sigma, \pi}\leftrightarrow R_{\widetilde\sigma, \widetilde\pi}$ 
	is not induced by an RSK-commuting bijection ${\sf Cont}_{\sigma, \pi}\leftrightarrow {\sf Cont}_{\widetilde\sigma, \widetilde\pi}$. 
	For example, let $(\sigma, \pi) = (34, 322)$ and $(\widetilde\sigma, \widetilde\pi) = (44, 431)$. Via explicit computation or Proposition~\ref{prop:triangular} one finds
	\[\hspace{-.2in}{\sf RSK}_{\sigma, \pi} =
	\begin{bmatrix}
		1 & 1 & 1 & 1 & 1 & 1 & 1 & 1\\
		0 & -1 & 0 & -2 & -1 & 0 & -2 & -1\\
		0 & 0 & -1 & 0 & -1 & -2 & -1 & -2\\
		0 & 0 & 0 & 1 & 0 & 0 & 1 & 0\\
		0 & 0 & 0 & 0 & 1 & 0 & 2 & 2\\
		0 & 0 & 0 & 0 & 0 & 1 & 0 & 1\\
		0 & 0 & 0 & 0 & 0 & 0 & -1 & 0\\
		0 & 0 & 0 & 0 & 0 & 0 & 0 & -1
	\end{bmatrix}, 
	\   {\sf RSK}_{\widetilde\sigma, \widetilde\pi} = 
	\begin{bmatrix}
		1 & 1 & 1 & 1 & 1 & 1 & 1 & 1\\
		0 & -1 & 0 & -2 & -1 & -3 & -2 & -3\\
		0 & 0 & -1 & 0 & -1 & 0 & -1 & -1\\
		0 & 0 & 0 & 1 & 0 & 3 & 1 & 3\\
		0 & 0 & 0 & 0 & 1 & 0 & 2 & 3\\
		0 & 0 & 0 & 0 & 0 & -1 & 0 & -1\\
		0 & 0 & 0 & 0 & 0 & 0 & -1 & -3\\
		0 & 0 & 0 & 0 & 0 & 0 & 0 & 1
	\end{bmatrix}.\]
	Both matrices are diagonalizable by Proposition~\ref{prop:tridiag}, with characteristic polynomial 
	\[p_{{\sf RSK}_{\sigma,\pi}}=p_{{\sf RSK}_{{\widetilde \sigma},{\widetilde \pi}}}=(t-1)^4(t+1)^4,\] 
	so they must be similar. However, the change-of-basis matrix 
	\[\psi = \begin{bmatrix}
	1 & 0 & 0 &  0 &  0 &  -\frac{2}{3} &  0 &  0\\ 
	0 & 0 & 1 & -1 &  0 &  1 &  0 & 0\\
	-1& 0 & 0 & -1 & 0 & 1 & 0 & 0\\
 	0 & 0 &  0 & 0 &  0 &  1 &  0 &  0\\ 
	3 & 3 & 0 & 5 & 1 & -2 & -1& 0\\
	-1 & 0 & -1 &  0 &  0 & -1 & -1&  \frac{1}{3}\\
 	-3 & -3 & 0 & -3 & 0 & 0 &  3& 0\\
 	2 & 1 & 1 & 1 & 0 & 1 & 0 & 0
 	\end{bmatrix}\]
	witnessing similarity ($\psi\cdot {\sf RSK}_{\sigma,\pi}\cdot \psi^{-1}={\sf RSK}_{\widetilde\sigma,\widetilde\pi}$) is not a permutation matrix, 
	so there is no RSK-commuting bijection ${\sf Cont}_{\sigma, \pi}\leftrightarrow{\sf Cont}_{\widetilde\sigma, \widetilde\pi}$.
\end{remark}

\section{Determinant}\label{sec:det}
We consider the determinant of ${\sf RSK}_{m, n, d}$, which is always $\pm1$ by Proposition~\ref{prop:det}.
The table below computes $\det {\sf RSK}_{m,n,d}$ for low values of $d$ and $m = n$. 

\begin{table}[h]
\begin{tabular}{c|*{9}{c}}
$m$\textbackslash $d$ & $1$ & $2$ & $3$ & $4$ & $5$ & $6$ & $7$ & $8$ & $9$ \\
\hline
$1$ & $1$ & $1$ & $1$ & $1$ & $1$ & $1$ & $1$ & $1$ & $1$ \\
$2$ & $1$ & $-1$ & $1$ & $-1$ & $1$ & $1$ & $1$ & $1$ & $1$ \\
$3$ & $1$ & $-1$ & $-1$ & $1$ & $1$ & $-1$ & $-1$ & $\ddots$ & \\
$4$ & $1$ & $1$ & $1$ & $1$ & $1$ & $\ddots$ & & & \\
$5$ & $1$ & $1$ & $1$ & $1$ & & & & & \\
\end{tabular}
\caption{Values of $\det {\sf RSK}_{m,m,d}$}\label{table:det}
\end{table}

The following theorem is a consequence of Theorem~\ref{thm:blockbuild} and allows us to determine $\det {\sf RSK}_{m, n, d}$ from determinants of the blocks ${\sf RSK}_{\sigma, \pi}$ for reduced pairs $(\sigma, \pi)$.

\begin{theorem}\label{thm:detform}
	For all $m, n, d\in\naturals$ we have the formula
	\[\det {\sf RSK}_{m, n, d} = \prod_{(\sigma, \pi)}(\det {\sf RSK}_{\sigma, \pi})^{N_{\sigma, \pi}(m, n, d)},\]
	where the product is over nonzero reduced pairs $(\sigma, \pi)$ of degree $d'\leq d$ and $N_{\sigma, \pi}(m, n, d)$ is as in Theorem~\ref{thm:blockbuild}. 
	Furthermore, when $m = n$ this formula simplifies to 
	\[\det {\sf RSK}_{m, n, d} = \prod_\sigma \det {\sf RSK}_{\sigma, \sigma},\]
	where the product is over $\sigma$ such that $(\sigma, \sigma)$ is reduced of degree $d'\leq d$ and $A_{\sigma, \pi}(d){m\choose\ell(\sigma)}$ is odd.
\end{theorem}
\begin{proof}
	Since the determinant of a block diagonal matrix is the product of the determinants of the blocks and the determinant of the $1\times 1$ identity matrix is clearly $1$, 
	the first formula is immediate from Theorem~\ref{thm:blockbuild}. To derive the second formula from the first, note that since $\det {\sf RSK}_{\sigma, \pi} = \pm 1$ by Proposition~\ref{prop:det}, 
	$(\det {\sf RSK}_{\sigma, \pi})^{N_{\sigma, \pi}(m, n, d)}$ is $1$ if $N_{\sigma, \pi}(m, n, d)$ is even and $\det {\sf RSK}_{\sigma, \pi}$ otherwise. 
	In the special case where $m = n$, we have
	\[N_{\sigma, \pi}(m, m, d) := \begin{cases}
		2A_{\sigma, \pi}(d){m\choose \ell(\sigma)}{m\choose \ell(\pi)} & \text{if } \sigma\neq \pi,\\
        		A_{\sigma, \pi}(d){m\choose \ell(\sigma)}^2 & \text{if } \sigma = \pi.
    		\end{cases}\]
	Thus $N_{\sigma, \pi}(m, m, d)$ is odd if and only if $\sigma = \pi$ and both $A_{\sigma, \pi}(d)$ and ${m\choose\ell(\sigma)}$ are odd.
\end{proof}

\begin{remark}
	When computing $\det {\sf RSK}_{m, n, d}$ for $m = n$, the parity of $A_{\sigma, \sigma}(d)$ is easy to check. 
	Indeed, if $(\sigma, \sigma)$ is reduced of degree $d' > 0$ and $\ell(\sigma) > 2$ then $g_{\sigma, \sigma}\leq 1$, so we find that $A_{\sigma, \sigma}(d)$ is congruent to $\delta_{d, d'}$ mod $2$ 
	(if $\ell(\sigma) = 2$ or $g_{\sigma, \sigma} = 0$) or $1$ (otherwise). 
\end{remark}

\begin{theorem}[Periodicity]\label{thm:detperiod}
	Fix $d$ and let $r$ be the least positive integer such that $2^r > d$. Then $\det {\sf RSK}_{m, n, d}$ has period $2^r$ in both $m$ and $n$, i.e., for all $m$ and $n$ we have
	\[\det {\sf RSK}_{m, n, d} = \det {\sf RSK}_{m+2^r, n, d} = \det {\sf RSK}_{m, n+2^r, d}.\] 
\end{theorem}
\begin{proof}
	Let $(\sigma, \pi)$ be a nonzero reduced pair of degree $d' \leq d$ and let $f_{\sigma, \pi}(m, n)$ be the function recording the parity of $N_{\sigma, \pi}(m, n, d)$. 
	By the first formula in Theorem~\ref{thm:detform} and the fact that $\det {\sf RSK}_{\sigma, \pi} = \pm1$ by Proposition~\ref{prop:det}, 
	it suffices to show that $f_{\sigma, \pi}(m, n)$ is periodic of period $2^r$ in both $m$ and $n$. Recall from Theorem~\ref{thm:blockbuild} that
	\[N_{\sigma, \pi}(m, n, d) := \begin{cases}
		A_{\sigma, \pi}(d)\left({m\choose \ell(\sigma)}{n\choose \ell(\pi)}+ {m\choose \ell(\pi)}{n\choose \ell(\sigma)}\right) & \text{if } \sigma\neq \pi,\\
        		A_{\sigma, \pi}(d){m\choose \ell(\sigma)}{n\choose \ell(\pi)} & \text{if } \sigma = \pi.
    		\end{cases}\]
	Let 
	\[\widetilde{f}_{\sigma, \pi}(m, n) := \begin{cases}
		\left({m\choose \ell(\sigma)}{n\choose \ell(\pi)}+ {m\choose \ell(\pi)}{n\choose \ell(\sigma)}\right)\text{ mod } 2 & \text{if } \sigma\neq \pi,\\
        		{m\choose \ell(\sigma)}{n\choose \ell(\pi)}\text{ mod } 2 & \text{if } \sigma = \pi.
    		\end{cases}\]
	Since $A_{\sigma, \pi}(d)$ is independent of $m$ and $n$, it suffices to show that $\widetilde{f}_{\sigma, \pi}(m, n)$ has the desired periodicity. 
	For all $a, b\geq 0$, Lucas's theorem implies a binomial coefficient ${a\choose b}$ is odd if and only if the $1$'s in the binary expansion of $b$ are a subset of the
	$1$'s in the binary expansion of $a$. 
	When $b$ is fixed and has $s$ digits in its binary expansion, this mean that the parity of ${a\choose b}$ depends only on the last $s$ digits of the binary expansion of $a$. 
	Hence the parity of ${a \choose b}$ is a periodic function of period $2^s$ in $a$. 
	It follows that $\widetilde{f}_{\sigma, \pi}(m, n)$ is periodic of period $2^s$ in both $m$ and $n$, where $s$ is minimal such that $2^s > \ell(\sigma), \ell(\pi)$. 
	Since the degree-$d$ reduced pair maximizing $\ell(\sigma)$ and $\ell(\pi)$ is $(1^d, 1^d)$, it follows that for any reduced pair $(\sigma, \pi)$ of degree $d'\leq d$, 
	the function $\widetilde{f}_{\sigma, \pi}(m, n)$ has period $2^r$ with $r$ as in the corollary statement. This completes the proof.
\end{proof}

\begin{example}
By Theorem~\ref{thm:detform}, $\det {\sf RSK}_{2^k,2^k,d}=1$ for any $d<2^k$. Indeed, by Lucas's theorem, for each 
reduced weight pair $(\sigma,\sigma)$, $N_{\sigma,\sigma}(2^k,2^k,d)$ is even since $\ell(\sigma)\leq d<2^k$.
\end{example}

\begin{example}\label{exa:dupto3det}
	We fully determine $\det {\sf RSK}_{m, n, d}$ for $d\leq 3$ using Theorem~\ref{thm:detform} and Example~\ref{exa:dupto3}. 
	The case $d = 1$ is easy: ${\sf RSK}_{m, n, 1} = {\rm Id}_{mn}$ so $\det {\sf RSK}_{m, n, 1} = 1$. 
	When $d = 2$, the only nontrivial block that occurs in ${\sf RSK}_{m, n, 2}$ is ${\sf RSK}_{11, 11}$, which has determinant $-1$ and appears with multiplicity
	$N_{11, 11}(m, n, 2) = 1\cdot{m\choose 2}{n\choose 2}$. By evaluating the parity of these binomial coefficients via Lucas's theorem, we compute 
	\[\det {\sf RSK}_{m, n, 2} = \begin{cases}
		-1 & \text{if } m, n\equiv 2 \text{ or } 3\ (\text{mod } 4),\\
		1 & \text{otherwise}.
		\end{cases}\]
	When $d=3$, there are four nontrivial blocks that can occur in ${\sf RSK}_{m, n, 3}$, namely those indexed by the reduced pairs $(11, 11)$, $(21, 111)$, $(12, 111)$, and $(111, 111)$. 
	Direct computation shows that $\det {\sf RSK}_{21, 111} = \det {\sf RSK}_{12, 111} = 1$ and $\det {\sf RSK}_{11, 11} = \det {\sf RSK}_{111, 111} = -1$.  
	Since $A_{11, 11}(3) = 4$ and $A_{111, 111}(3) = 1$, we conclude that 
	\[\det {\sf RSK}_{m, n, 3} = \begin{cases}
		-1 & \text{if } m, n\equiv 3\  (\text{mod } 4),\\
		1 & \text{otherwise}.
		\end{cases}\]
	The reader can compare these formulas to the first three columns of Table~\ref{table:det}. 
\end{example}

\begin{remark}
	The minimal period of $\det {\sf RSK}_{m, n, d}$ for a fixed $d$ may be lower than the value guaranteed in Theorem~\ref{thm:detperiod}. For instance, the proof of Theorem~\ref{thm:detperiod} 
	demonstrates that $\det {\sf RSK}_{m, n, 4}$ has period $8$ because the parity of $N_{1111, 1111}(m, n, 4)$ is a function of minimal period $8$ by Lucas's theorem. 
	This turns out to be irrelevant, since $\det {\sf RSK}_{1111, 1111} = 1$. 
	Thus the minimal period of $\det {\sf RSK}_{m, n, 4}$ is $4$ rather than $8$.
\end{remark}

\begin{example}\label{exa:2x2det}
	We fully determine $\det {\sf RSK}_{2, 2, d}$ using Theorem~\ref{thm:detform}. By Theorem~\ref{thm:blockbuild}, we know that the only nontrivial blocks appearing in ${\sf RSK}_{2, 2, d}$ are those corresponding to
	reduced pairs of the form $(aa, aa)$ for some $a > 0$. The explicit description of the eigenvalues of ${\sf RSK}_{aa, aa}$ in Example~\ref{exa:2x2eigen} shows that $\det {\sf RSK}_{aa, aa} = -1$ if and only if 
	$\lfloor \frac{a+1}{2}\rfloor$ is odd, which occurs if and only if $a$ is congruent to $1$ or $2$ mod $4$. Moreover, $A_{aa, aa}(d) = 4(d-2a)+\delta_{d, 2a}$ is odd if and only if $d = 2a$, 
	and ${m \choose \ell(aa)} = {2\choose 2}$ is odd. Combining these facts shows that
	\[\det {\sf RSK}_{2, 2, d} = \begin{cases}
		-1 & \text{if \ } d\equiv 2, 4\  (\text{mod } 8),\\
		1 & \text{otherwise}.
		\end{cases}\]
\end{example}

For any fixed $d$, Theorem~\ref{thm:detform} allows one to in principle determine $\det {\sf RSK}_{m, m, d}$ by computing $\det {\sf RSK}_{\sigma, \sigma}$ for a finite collection of 
weights $\sigma$ and combining them via a (somewhat messy) periodic formula using the binary expression of $m$. One can also compute $\det {\sf RSK}_{m, n, d}$ in this manner, 
although in the general case the number of pairs $(\sigma, \pi)$ contributing to the formula is much greater than when $m = n$.

\section{Trace}\label{sec:trace}
We consider the trace of ${\sf RSK}_{m,n,d}$. Data for small $d$ and $m=n$ is presented in Table~\ref{table:trace}.

\begin{table}[h]\label{tab:trace}
\begin{tabular}{c|*{9}{c}}
$m$\textbackslash $d$ & $1$ & $2$ & $3$ & $4$ & $5$ & $6$ & $7$ & $8$ & $9$ \\
\hline
$1$ & $1$ & $1$ & $1$ & $1$ & $1$ & $1$ & $1$ & $1$ & $1$ \\
$2$ & $4$ & $8$ & $12$ & $17$ & $24$ & $32$ & $40$ & $49$ & $60$ \\
$3$ & $9$ & $27$ & $42$ & $70$ & $160$ & $241$ & $203$ & $\ddots$ & \\
$4$ & $16$ & $64$ & $48$ & $-33$ & $613$ & $\ddots$ & & & \\
$5$ & $25$ & $125$ & $-175$ & $-1650$ & & & & & \\
\end{tabular}
\caption{Values of ${\rm Tr} \ {\sf RSK}_{m,m,d}$}
\label{table:trace}
\end{table}
As a consequence of our results, we show in Example~\ref{exa:2x2trace} that
the $m=2$ row agrees with the ``concentric square numbers'' (\url{https://oeis.org/A194274}). The $m=3,4,5$ rows do not match anything in OEIS at the time of writing. 

\begin{theorem}\label{thm:tracepoly}
	For fixed $d$, ${\rm Tr}\ {\sf RSK}_{m, n, d}$ is a polynomial in $O(m^dn^d)$. More specifically,  
	\[{\rm Tr}\ {\sf RSK}_{m, n, d} = N_0(m, n, d)+\sum_{(\sigma, \pi)}N_{\sigma, \pi}(m, n, d) {\rm Tr}\ {\sf RSK}_{\sigma, \pi},\]
	where the sum is over nonzero reduced pairs $(\sigma, \pi)$ of degree $d'\leq d$ and the expressions $N_{\sigma, \pi}(m, n, d)$ and $N_0(m, n, d)$ are as in Theorem~\ref{thm:blockbuild}.
	The lead term of this polynomial is 
	\[\frac{{\rm Tr}\ {\sf RSK}_{1^d, 1^d}}{(d!)^2}m^dn^d\]
	whenever ${\rm Tr}\ {\sf RSK}_{1^d, 1^d}\neq 0$.
\end{theorem}
\begin{proof}
	The first formula in the theorem statement is immediate from the block decomposition of ${\sf RSK}_{m, n, d}$ given in Theorem~\ref{thm:blockbuild}. 
	The formula for $N_{\sigma, \pi}(m, n, d)$ in Theorem~\ref{thm:blockbuild} shows that for fixed $d$, $N_{\sigma, \pi}(m, n, d)$ is a polynomial in $m$ and $n$ whose lead term has total degree $\ell(\sigma)+\ell(\pi)$. 
	Similarly, the expression for $N_0(m, n, d)$ is a polynomial whose lead term has total degree $d+1$.  
	It follows that ${\rm Tr}\ {\sf RSK}_{m, n, d}$ is polynomial in $m$ and $n$ for any fixed $d$. Moreover, the term in the sum with the largest possible 
	growth in $m$ and $n$ comes from the reduced pair $(1^d, 1^d)$, which maximizes $\ell(\sigma)$ and $\ell(\pi)$. The lead term of $N_{1^d, 1^d}(m, n, d)$ is $\frac{(mn)^d}{(d!)^2}$, and the result follows.
\end{proof}

\begin{example} 
	We compute ${\rm Tr}\ {\sf RSK}_{m, n, d}$ for $d\leq 3$ using Theorem~\ref{thm:tracepoly} and Example~\ref{exa:dupto3}. When $d=1$, we have ${\rm Tr}\ {\sf RSK}_{m, n, 1} = mn$ 
	since ${\sf RSK}_{m, n, 1} = {\rm Id}_{mn}$.  
	When $d = 2$, Theorem~\ref{thm:tracepoly} and the values $N_{\sigma, \pi}(m, n, d)$ from Example~\ref{exa:dupto3} give the formula
	\[{\rm Tr}\ {\sf RSK}_{m, n, 2} =  N_0(m, n, 2)+{\rm Tr}\ {\sf RSK}_{11, 11}{m\choose 2}{n\choose 2}.\]
	Evaluating $N_0(m, n, 2)$ and observing from Example~\ref{exa:11-11} that ${\rm Tr}\ {\sf RSK}_{11, 11} = 0$ shows that 
	\[{\rm Tr}\ {\sf RSK}_{m, n, 2} = \frac{mn(n+1)+nm(m+1)-2mn}{2},\]
	which simplifies to $m^3$ in the special case where $m = n$.

	The $d = 3$ case is similar. Computing that 
	\[{\rm Tr}\ {\sf RSK}_{21, 111} = 0, {\rm Tr}\ {\sf RSK}_{12, 111} = -1, \text{and ${\rm Tr}\  {\sf RSK}_{111, 111} = -3$}\] 
	implies the following formula, using the values of $N_{\sigma, \pi}(m, n, d)$ from Example~\ref{exa:dupto3}:
	\[{\rm Tr}\ {\sf RSK}_{m, n, 3} = \left({{n+2}\choose 3}m+{{m+2}\choose 3}n-mn\right)-3{m\choose 3}{n\choose 3} - \left({m\choose 2}{n\choose 3}+{m\choose 3}{n\choose 2}\right).\]
	When $m = n$ this formula simplifies to 
	\[{\rm Tr} \ {\sf RSK}_{m,m,3}=-\frac{m^2(m^4-4m^3+m^2-14m+4)}{12}.\]
	The reader can compare these formulas to the first three columns of Table~\ref{table:trace}. 
\end{example}

\begin{example}\label{exa:2x2trace}
	Theorem~\ref{thm:tracepoly} and Example~\ref{exa:2x2eigen} suffice to give a simple formula for ${\rm Tr}\ {\sf RSK}_{2, 2, d}$, explaining the connection with concentric square numbers mentioned above. 
	In Example~\ref{exa:2x2block} we computed that 
	\[{\sf RSK}_{2, 2, d} = ({\rm Id}_1^{\oplus 4d})\oplus\left(\bigoplus_{a = 1}^{\lfloor d/2\rfloor}{\sf RSK}_{aa, aa}^{\oplus 4(d-2a)+\delta_{d, 2a}}\right).\]
	Example~\ref{exa:2x2eigen} shows that ${\rm Tr}\ {\sf RSK}_{aa, aa}$ is $0$ for $a$ odd and $1$ for $a$ even. 
	This yields a trace formula matching the recurrence $s_d = 4d + s_{d-4}$ for the concentric square numbers $(s_d)$:
	\[{\rm Tr}\ {\sf RSK}_{2, 2, d} = 4d+\sum_{b=1}^{\lfloor d/4\rfloor}(4(d-4b)+\delta_{d, 4b}).\]
\end{example}

Theorem~\ref{thm:tracepoly} shows that formulas for ${\rm Tr}\ {\sf RSK}_{1^d, 1^d}$ determine the growth of ${\rm Tr} \ {\sf RSK}_{m,n,d}$
except in the ``accidental'' cases (such as $d=2$) where ${\rm Tr}\ {\sf RSK}_{1^d, 1^d}=0$. We see 
\[\lim_{m, n\to\infty} {\rm Tr} \ {\sf RSK}_{m,n,d}\to-\infty \text{ \ if $d=3,4$}\] 
and yet 
\[\lim_{m, n\to\infty} {\rm Tr} \ {\sf RSK}_{m,n,d}\to \infty \text{\ if $d=5$}.\]

What are the values (or even the signs) of ${\rm Tr}\ {\sf RSK}_{1^d, 1^d}$ for $d\geq 1$? In particular, is ${\rm Tr}\ {\sf RSK}_{1^d, 1^d}$ always positive (or always negative) for $d$ sufficiently large? To better investigate ${\rm Tr}\ {\sf RSK}_{1^d, 1^d}$, we record a practical formula that follows immediately from Proposition~\ref{prop:permweight}. 
It allowed us to compute up to $d = 11$:
\[\{{\rm Tr} \ {\sf RSK}_{1^d,1^d}\}_{d\geq 1}= \{1,0,-3,-5,23,96,-279,-3498,124, 120819, 185838, \ldots\}.\]

\begin{corollary}\label{cor:thetracealg}
 	\[{\rm Tr} \ {\sf RSK}_{1^d,1^d}= \sum_{\alpha\in{\sf Cont}_{1^d, 1^d}} \prod_c \det \alpha_c(\alpha).\]
\end{corollary}

\begin{conjecture}
${\rm Tr}\ {\sf RSK}_{1^d, 1^d}\neq 0$ for $d\neq 2$, i.e., ${\rm Tr} \ {\sf RSK}_{m,n,d}$ has total degree $2d$ for $d\neq 2$.
\end{conjecture}
In order to compute ${\rm Tr}\ {\sf RSK}_{m, n, d}$ more quickly, one must identify more efficient formulas for computing the diagonal entries ${\sf RSK}_{\sigma, \pi}(\alpha, \alpha)$. 
We have such formulas for each of the families from Section~\ref{sec:examples} (see Proposition~\ref{prop:permweight}, Proposition~\ref{prop:rouB}, and Proposition~\ref{thm:trieigen}). 
In these cases we can compute ${\sf RSK}_{\sigma, \pi}(\alpha, \alpha)$ without determining the entire matrix, and it turns out that ${\sf RSK}_{\sigma, \pi}(\alpha, \alpha)\in \{-1, 0, 1\}$. 
However, we know of no formula to compute individual entries ${\sf RSK}_{\sigma, \pi}(\alpha, \alpha)$ for arbitrary weights $(\sigma, \pi)$, and our next result shows that these diagonal entries 
can be arbitrarily large. 

\begin{proposition}
	For all $N\in\naturals$ there exist $(\sigma, \pi)$ with $(\ell(\sigma), \ell(\pi)) = (2,3)$ and $\alpha\in{\sf Cont}_{\sigma, \pi}$ such that $|{\sf RSK}_{\sigma, \pi}(\alpha, \alpha)| = N$.
\end{proposition}
\begin{proof}
	We provide an explicit construction of $\sigma$, $\pi$, and $\alpha$. Let 
	\[\sigma = (N+2, N+1) \text{ \ and  \ } \pi = (N+1, N+1, 1).\]
	Further, let 
	$\alpha = \begin{bmatrix} 1 & n & 1\\ n & 1 & 0\end{bmatrix}\in{\sf Cont}_{\sigma, \pi}$.
	Then one can verify directly that
	\[{\sf RSK}(z^\alpha) = \begin{vmatrix} z_{11} & z_{12}\\ z_{21} & z_{22}\end{vmatrix}^N\begin{vmatrix}z_{11} & z_{13}\\ z_{21} & z_{23}\end{vmatrix}z_{12}.\]
	Since $z_{11}z_{23}$ is not a factor of $z^\alpha$, we have 
	\begin{align*}
		{\sf RSK}_{\sigma, \pi}(\alpha, \alpha) &= [z^\alpha]\left( \begin{vmatrix} z_{11} & z_{12}\\ z_{21} & z_{22}\end{vmatrix}^N\begin{vmatrix}z_{11} & z_{13}\\ z_{21} & z_{23}\end{vmatrix}z_{12}\right)\\
		&= [z_{11}z_{12}^{N-1}z_{21}^{N-1}z_{22}]\left(\begin{vmatrix}z_{11} & z_{12}\\ z_{21} & z_{22}\end{vmatrix}^N\right) \cdot 
			[z_{21}z_{13}z_{12}]\left(\begin{vmatrix}z_{11} & z_{13}\\ z_{21} & z_{23}\end{vmatrix}z_{12}\right)\\
		&= (-1)^{N-1}{N\choose 1}\cdot (-1)\\
		&= (-1)^NN.\qedhere
	\end{align*}
\end{proof}

Even for permutation weights, one may hope for a more efficient formula to determine when diagonal entries of ${\sf RSK}_{1^d, 1^d}$ are $0$. 
Let $C_d$ denote the set of $\alpha\in{\sf Cont}_{1^d, 1^d}$ such that the diagonal entry ${\sf RSK}_{1^d, 1^d}(\alpha, \alpha) = 0$. 
We computed that,
\[\{|C_d|\}_{d\geq 1} = \{0,0,1,7,53,406,3373,30360, 297256, 3153559, 36186708,\dots\}.\]
\begin{conjecture}\label{conj:diagto0}
$\lim_{d\to\infty} |C_d|/d! = 1$.
\end{conjecture}
By Corollary~\ref{cor:thetracealg}, to prove Conjecture~\ref{conj:diagto0} it suffices to show the following. Under RSK insertion of almost any permutation matrix (i.e., Schensted insertion),
some label in the first row of the partially completed $P$ tableau gets bumped to a column strictly to the left.

Since ${\sf RSK}_{\sigma, \pi}$ and ${\sf RSK}_{\widetilde\sigma, \widetilde\pi}$ are similar if and only if their inverses are, a result analogous to Theorem~\ref{thm:blockbuild} holds for the matrices ${\sf RSK}^{-1}_{m, n, d}$. 
This implies similar results on the trace, determinant, and so on. We provide trace data for ${\sf RSK}^{-1}_{m, m, d}$ for the interested reader.

\begin{table}[h]
\begin{tabular}{c|*{9}{c}}
$m$\textbackslash $d$ & $1$ & $2$ & $3$ & $4$ & $5$ & $6$ & $7$ & $8$ & $9$ \\
\hline
$1$ & $1$ & $1$ & $1$ & $1$ & $1$ & $1$ & $1$ & $1$ & $1$ \\
$2$ & $4$ & $8$ & $12$ & $17$ & $24$ & $32$ & $40$ & $49$ & $60$ \\
$3$ & $9$ & $27$ & $44$ & $64$ & $118$ & $185$ & $201$ & $\ddots$ & \\
$4$ & $16$ & $64$ & $80$ & $-29$ & $24$ & $\ddots$ & & & \\
$5$ & $25$ & $125$ & $25$ & $-1250$ & & & & & \\
\end{tabular}
\caption{Values of ${\rm Tr} \ {\sf RSK}_{m,m,d}^{-1}$}
\end{table}

Note that ${\rm Tr}\ {\sf RSK}^{-1}_{2, 2, d} = {\rm Tr}\ {\sf RSK}_{2,2, d}$; this is because all of the reduced pairs $(\sigma, \pi)$ contributing to the trace in this special case of 
Theorem~\ref{thm:tracepoly} are triangular and therefore self-inverse by the calculation in the proof of Proposition~\ref{prop:tridiag}.

\section{Tables for reduced pairs \texorpdfstring{$(\sigma,\pi)$}{}}\label{sec:tables}

We compile data for $M={\sf RSK}_{\sigma,\pi}$ where 
$(\sigma,\pi)$ is reduced with 
$\ell(\sigma)\in \{2,3\}, \ell(\pi)=3$ and $d=|\sigma|=|\pi|\in \{3,4,5\}$.
The data
for any $(\sigma,\pi)$ with $\ell(\sigma),\ell(\pi)\leq 3$ and $d\leq 5$ can be deduced from these tables via 
Corollary~\ref{cor:redunique} and Lemma~\ref{lemma:easyRSKcommuting}. The case $\ell(\sigma)=\ell(\pi)=2$ is covered by Example~\ref{exa:2x2}. Finally, if $\ell(\sigma)=1$, then the only reduced pair is $(0, \vec{0})$ by Corollary~\ref{cor:redweuse}, and ${\sf RSK}_{0,\vec{0}}={\rm Id}_1$.

\begin{table}[h!]
\centering
\[
\begin{array}{|c|c|c|c|l|}
\hline
\sigma & \pi  & \det \ M    & {\rm Tr} \  M  & p_M(t) \\
\hline
21 & 111 & 1 & 0  & (t-1)(t^2+t+1)\\ 
\hline
12 & 111 & 1 & -1 & (t-1)(t+1)^2\\
\hline
111 & 111 & -1 & -3 & (t - 1)(t + 1)^2 (t^3 + 2t^2 + 1) \\
\hline
\end{array}
\]
\caption{$d=3$}
\end{table}

\begin{table}[h!]
\centering
\[
\begin{array}{|c|c|c|c|l|}
\hline
\sigma & \pi & \det  \ M & {\rm Tr} \ M & p_{M}(t) \\
\hline
22 & 211 & 1 & 0 & (t-1)^2(t+1)^2\\
\hline
22 & 121 & 1 & 1 & (t-1)(t+1)(t^2-t-1)\\
\hline
22 & 112 & 1 & 1 & (t-1)^2(t^2+t+1) \\
\hline
211 & 211 & -1 & -2 & (t - 1)^2 (t + 1)^2 (t^3 + 2t^2 + 1) \\
121 & 121 &\ &\ &\ \\
112 & 112 &\ & \ & \ \\
\hline
211 & 121 & -1 & -1 & (t - 1)(t + 1)(t^5 + t^4 - 3t^3 - 2t^2 - t - 1) \\
\hline
211 & 112 & -1 & -1 & (t - 1)^2 (t + 1)(t^2 + t + 1)^2 \\
\hline
121 & 112 & -1 & -2 & (t - 1)(t + 1)^2 (t^4 + t^3 - 2t^2 - t - 1) \\
\hline
\end{array}
\]
\caption{$d=4$}
\end{table}

\begin{table}[h!]
\centering
\[
\begin{array}{|c|c|c|c|l|}
\hline
\sigma & \pi & \det \ M  & {\rm Tr} \ M  & p_M(t) \\
\hline
32 & 221 & 1 & 0 & (t-1)(t+1)(t^3+t+1)\\
\hline
32 & 212 & 1 & 1 & (t-1)(t^2-t+1)(t^2+t+1)\\
32 & 122 & \ & \ & \ \\
\hline
23 & 221 & 1 & 1 & (t-1)^3(t+1)^2\\
23 & 212 & \ & \ & \ \\
\hline
23 & 122 & 1 & 2 & (t-1)^2(t+1)(t^2-t-1)\\
\hline
311 & 221 & -1 & -1 & (t - 1)(t + 1)(t^6 + t^5 + 4t^3 + 2t^2 + t + 1) \\
\hline
311 & 212 & -1 & -1 & (t - 1)(t^2 + t + 1)^2(t^3 - t + 1) \\
\hline
311 & 122 & -1 & -1 & (t - 1)(t + 1)(t^2 - t + 1)(t^2 + t + 1)^2 \\
\hline
221 & 131 & -1 & -1 & (t - 1)(t + 1)(t^6 + t^5 - 2t^4 + 2t^3 + 2t^2 + 1) \\
\hline
212 & 131  & -1 & -1 & (t - 1)(t+ 1)^2(t^5 - 4t^3 + t^2 + 1) \\
\hline
221 & 113 & -1 & 0  & (t - 1)^3(t + 1)(t^2 + t + 1)^2 \\
212 & 113 & \ & \  & \ \\
\hline
221 & 221 & -1 & 1  & (t - 1)^3(t + 1)^2(t^3 - 2t^2 - 1)(t^3 + 2t^2 + 1) \\
212 & 212 & \ & \ & \ \\
122 & 122 & \ & \  & \ \\
\hline
221 & 212 & -1 & 1  & (t - 1)^3(t + 1)^2(t^6 - 4t^4 - 2t^3 - 3t^2 - t - 1) \\
\hline
221 & 122 & -1 & 1  & (t - 1)^2(t + 1)(t^2 - t - 1)(t^6 + t^5 - 2t^4 - t^3 - 2t^2 - t - 1) \\
\hline
212 & 122 & -1 & 2  & (t - 1)^2(t + 1)(t^3 - 2t^2 - 1)(t^5 + t^4 - 3t^3 - 2t^2 - t - 1) \\
\hline
131 & 122 & -1 & -1 & (t - 1)^2(t + 1)^2(t^4 + t^3 - 2t^2 - t - 1) \\
122 & 113 & \ & \ & \  \\
\hline
\end{array}
\]
\caption{$d=5$}
\end{table}

We remark that, while in our tables ${\rm det}\ {\sf RSK}_{\sigma, \pi} = -1$ for all pairs $(\sigma,\pi)$ with $\ell(\sigma)=\ell(\pi)=3$, this is not true in general. 
For example, ${\rm det}\ {\sf RSK}_{321, 321} = 1$. In this case the trace is $0$ and the characteristic polynomial is 
\[p_{{\sf RSK}_{321, 321}}(t)=(t-1)^3(t+1)^3(t^3-2t^2-1)(t^3+2t^2+1).\]
Among the $(\sigma,\pi)$ listed in the tables, the sublist 
for which ${\sf RSK}_{\sigma,\pi}$ is not diagonalizable are
\[(211,211), (121,121), (212,212), (122,122), (221,212), (212,122).\]

\section*{Acknowledgements}
We thank  Shiliang Gao for his explanations to us about Standard Monomial Theory; in turn, they led us to think about our problem set. We are very grateful to Arianna Doran for her coding work as part of the ICLUE program at UIUC. 
We thank Sergey Fomin for suggesting to us the Dynkin diagram formulation of our diagonalizability result, Theorem~\ref{thm:sampler}(I). 
We also acknowledge Nathan Hayes, Abigail Price, Victor Reiner, and David Xia for helpful conversations. We made use of {\tt Macaulay2} and {\tt SageMath} in our experimentation.
AS was supported by the NSF graduate fellowship.
AY was supported by a Simons Collaboration grant. Both authors were partially supported by an NSF RTG in Combinatorics (DMS 1937241).


\begin{thebibliography}{99}

\bibitem{Ariki}
Ariki, Susumu. Robinson-Schensted correspondence and left cells.  in {\it Combinatorial methods in representation theory (Kyoto, 1998)}, 1--20, Adv. Stud. Pure Math., 28, Kinokuniya, Tokyo.

\bibitem{Brualdi}
Brualdi, Richard A. Introductory combinatorics. Fifth edition. Pearson Prentice Hall, Upper Saddle River, NJ, 2010. xii+605 pp. 

\bibitem{Bruns}
Bruns, Winfried; Conca, Aldo; Raicu, Claudiu; Varbaro, Matteo. Determinants, Gr\"obner bases and cohomology. Springer Monographs in Mathematics. Springer, Cham, 2022. xiii+507 pp.

\bibitem{Procesi}
de Concini, Corrado; Procesi, Claudio. A characteristic free approach to invariant theory. Advances in Math. 21 (1976), no. 3, 330--354.

\bibitem{Diaconis}
Diaconis, Persi; Gangolli, Anil. Rectangular arrays with fixed margins. In Discrete Probability and Algorithms (D. Aldous, P. Diaconis, J. Spencer, and J.M. Steele, eds.) Springer, New York, 1995.

\bibitem{Rota}
Doubilet, Peter; Rota, Gian-Carlo; Stein, Joel. On the foundations of combinatorial theory. IX. Combinatorial methods in invariant theory. Studies in Appl. Math. 53 (1974), 185--216. 

\bibitem{Dyer} Dyer, Martin; Kannan, Ravi; Mount, John. Sampling contingency tables. Random Structures Algorithms 10 (1997), no. 4, 487--506. 

\bibitem{Fulton}
Fulton, William. Young tableaux. With applications to representation theory and geometry. London Mathematical Society Student Texts, 35. Cambridge University Press, Cambridge, 1997. {\rm x}+260 pp

\bibitem{Howe}
Howe, Roger. Perspectives on invariant theory: Schur duality, multiplicity-free actions and beyond. The Schur lectures (1992) (Tel Aviv), 1--182, Israel Math. Conf. Proc., 8, Bar-Ilan Univ., Ramat Gan, 1995.

\bibitem{KL}
Kazhdan, David; Lusztig, George. Representations of Coxeter groups and Hecke algebras. Invent. Math. 53 (1979), no. 2, 165--184.

\bibitem{Ram}
Ram, Arun. An elementary proof of Roichman's rule for irreducible characters of Iwahori-Hecke algebras of type A, in {\it Mathematical essays in honor of Gian-Carlo Rota (Cambridge, MA, 1996)}, 335--342, Progr. Math., 161, Birkh\"auser Boston, Boston, MA.

\bibitem{Serre}
Serre, Jean-Pierre. Linear representations of finite groups. Translated from the second French edition by Leonard L. Scott. Graduate Texts in Mathematics, Vol. 42. Springer-Verlag, New York-Heidelberg, 1977. {\rm x}+170 pp.

\bibitem{Sidorov}
Sidorov, S.~V. Similarity of matrices with integer spectrum over the ring of integers. Russian Math. (Iz. VUZ) {\bf 55} (2011), no.~3, 77--84; translated from Izv. Vyssh. Uchebn. Zaved. Mat. {\bf 2011}, no.~3, 86--94.

\bibitem{ECII}
Stanley, Richard P. Enumerative combinatorics. Vol. 2. Second edition [of 1676282]. With an appendix by Sergey Fomin. Cambridge Studies in Advanced Mathematics, 208. Cambridge University Press, Cambridge, [2024], 2024. xvi+783 pp.

\end{thebibliography}
\end{document}